\documentclass[11pt,a4paper,reqno]{amsart}%
\usepackage{amsthm,amsmath,amsfonts,amssymb,xcolor,amsxtra,appendix,bookmark,dsfont,bm,mathrsfs}
\usepackage[latin1]{inputenc}
\usepackage{color} 
\usepackage{amssymb}\usepackage{graphicx}
\usepackage{tikz}
\usepackage{hyperref}
\theoremstyle{plain}
\newtheorem{theorem}{Theorem}
\newtheorem{lemma}[theorem]{Lemma}

\newtheorem{proposition}[theorem]{Proposition}

\newtheorem{definition}{Definition}

\theoremstyle{remark}
\newtheorem{remark}[theorem]{Remark}

\allowdisplaybreaks

\title[ Formation of   condensations  ]{Formation of  condensations  
	 for 
	 non-radial solutions to 3-wave kinetic equations}

\author[G. Staffilani]{Gigliola Staffilani
}
\address{Department of Mathematics, Massachusetts Institute of Technology, Cambridge, MA 02139, USA}
\email{gigliola@math.mit.edu} 
\thanks{G.S. is  funded in part by  the NSF grants DMS-2052651, DMS-2306378 and the Simons Foundation through the Simons Collaboration on Wave Turbulence.}

\author[M.-B. Tran]{Minh-Binh Tran}
\address{Department of Mathematics, Texas A\&M University, College Station, TX 77843, USA}
\email{minhbinh@tamu.edu} 
\thanks{M.-B. T is  funded in part by  a   Humboldt Fellowship,   NSF CAREER  DMS-2303146, and NSF Grants DMS-2204795, DMS-2305523,  DMS-2306379.}

\begin{document}
\date{\today}

\begin{abstract}  
	We consider in this work a $2$-dimensional $3$-wave kinetic equation describing the dynamics of the thermal cloud outside a Bose-Einstein Condensate. We construct   global non-radial mild solutions for the equation. Those mild solutions are the summation of Dirac   masses on circles.  We prove that in each spatial direction, either Dirac masses  at the origin, which are the so-called Bose-Einstein condensates, can be formed in finite time or the solutions converge to Bose-Einstein condensates  as time evolves to infinity. We also describe a dynamics of the formation of the Bose-Einstein condensates latter case. In this case, on each direction, the solutions  accumulate around circles close to the origin at growth rates at least     linearly in time. 
\end{abstract}

\maketitle

\tableofcontents

\section{Introduction}\label{intro} 
\subsection{The 3-wave kinetic equation}
Quantum gases in lower dimensions  \cite{hohenberg1967existence,penrose1956bose} have garnered significant attention as model systems for exploring diverse phenomena. Compared to their three-dimensional counterparts, dimensionally reduced systems can display remarkably novel properties. Recent experiments have demonstrated the occurrence of Bose-Einstein condensations (BECs) in several one- and two-dimensional systems \cite{bishop1978study,fisher1988dilute,gorlitz2001realization,nelson1977universal,rychtarik2004two}. In this work, we are interested in the analysis of the following 3-wave kinetic equation, coming from the theory of $2$-dimensional low temperature trapped bose gases  \cite{PomeauBinh}

\begin{equation}\label{3wave}
	\begin{aligned} 
		\frac{\partial f}{\partial \tau}(\tau,k) &\ = \  Q[ f](\tau,k)\,, \ \ \ \   f(0,k)=f_{in}(k) \ \ \ k\in\mathbb{R}^2,\ \ \ \tau\in [0,\infty).
	\end{aligned}
\end{equation}
where
\begin{align}\label{E1}
	\begin{split}
		Q[f]:=&\int _{\mathbb R^2}\int _{\mathbb R^2} \text{d}k_1\text{d}k_2 \mathcal W (|k|,|k_1|,|k_2|) \delta \left(   {\omega (k)}  - {\omega (k_1)} - {\omega (k_2)}  \right)  \delta (k-k_1-k_2)\\
		&\hspace{-.5cm}\times\big[ f(k_1)f(k_2) -(f(k_1)+f(k_2))f(k) \big]\,,\\
		&-\int _{\mathbb R^2}\int _{\mathbb R^2} \text{d}k_1\text{d}k_2 2\mathcal W (|k|,|k_1|,|k_2|) \delta \left(   {\omega (k_1)}  - {\omega (k)} - {\omega (k_2)}  \right)  \delta (k_1-k-k_2)\\
		&\hspace{-.5cm}\times\big[ f(k)f(k_2) -(f(k)+f(k_2))f(k_1) \big]\,\\
	\end{split}
\end{align}

The collision kernel is given by
\begin{align}\label{E3}
	\begin{split}
		\mathcal W (|k|,|k_1|,|k_2|)  \ = \ [|k|+|k_1|+|k_2|]^{-1}\frac1\pi.	\end{split}
\end{align}

The dispersion relation $\omega (k)$ is given by the phonon dispersion relation \begin{align}\label{E2}\omega(k)=|k|.\end{align}

Existence and uniqueness of strong radial solutions to the 3-dimensional version of \eqref{3wave} have been studied in \cite{AlonsoGambaBinh}. Maxwellian lower bounds and the convergence to equilibrium and the energy cascade phenomenon for the 3-dimensional case has also been done in \cite{nguyen2017quantum,EscobedoBinh} and \cite{soffer2020energy}. Numerical schemes for the radial and 3-dimensional version have also been provided in \cite{das2024numerical,walton2023numerical,walton2024numerical,walton2022deep}. Important and deep studies on the radial case of  the 3-dimensional version, but for the linearized case and related models  have been done in \cite{cortes2020system,escobedo2023linearized1,escobedo2023linearized,EPV}.  However, all of the above  mentioned works are for radial solutions.

In our previous works \cite{staffilani2024condensation,staffilani2024energy}, we consider the    4-wave kinetic equation 
\begin{eqnarray}\label{4wave}
	\partial_\tau f &\ = \ & 	\mathfrak Q \left[ f\right],\ \ \ f(0,k)=f_{in}(k) ,\ \ \ k\in\mathbb{R}^3,\ \ \ \tau\in [0,\infty),\\ 
	\mathfrak Q \left[ f\right]  
	&\ =\ & \iiint_{\mathbb{R}^{9}}\mathrm{d}k_1\,\mathrm{d}k_2\,\mathrm{d}k_3 \mathfrak W(|k|,|k_1|,|k_2|,|k_3|)  \delta(k+k_1-k_2-k_3)\\
	& &\times\delta(\omega + \omega_1 -\omega_2 - \omega_3)[f_2f_3(f_1+f)-ff_1(f_2+f_3)] ,\nonumber
\end{eqnarray}
where $\omega,\omega_1,\omega_2,\omega_3$  is the shorthand notation for  $\omega(k), \omega(k_1), \omega(k_2), \omega(k_3)$, and $f,f_1,f_2,f_3$ is the shorthand notation for $f(k), f(k_1), f(k_2), f(k_3)$. Extending of  the pioneering work of  Escobedo and Velazquez  \cite{EscobedoVelazquez:2015:FTB,EscobedoVelazquez:2015:OTT}, we prove that the solution of  \eqref{4wave} forms  a delta function at the origin (Bose-Einstein condensates) as time evolves, and the energy of the solution moves towards high frequencies, under the assumptions that  (i) the kernel  $\mathfrak W(|k|,|k_1|,|k_2|,|k_3|)$ is not singular, (ii) $\omega(|k|)$ is of the general form but different from $|k|$ (one example considered in these works is $\omega(k)=|k|^\alpha$ with $1<\alpha\le 2$), and (iii) the solution is isotropic $f(t,k)=f(t,|k|)$. Motivated by the previous works  \cite{staffilani2024condensation,staffilani2024energy} (and  inspired by  \cite{EscobedoVelazquez:2015:FTB,EscobedoVelazquez:2015:OTT}), we aim to remove those assumptions, but in the context of the 3-wave kinetic equation \eqref{3wave}. It is therefore the goal of this work to construct non-radial solutions of \eqref{3wave}, in which   the kernel \eqref{E3} is singular and the dispersion relation is of the phonon type \eqref{E2}.  We put the initial condition $f_{in}$ on circles   (see Definition \ref{Measure:4} for the definition of the circular lattice) and the origin is not in the support of  $f_{in}$. We show that as time evolves, the solutions remain to be  summations of Dirac masses on the circular lattice.  Moreover, either the solutions form   Bose-Einstein condensates  in finite time, or when time goes to infinity, they converge to Dirac functions at the origin   in all spatial directions, as they are non-radial. We are also able to  describe a dynamics of the formation of those Bose-Einstein condensates in the latter case: The concentrations of the solutions around circles close to the origin for each direction are growing at least linearly in time (see Remark \ref{R1}). To the best of our knowledge, this work appears to be the only available example for the formation of Bose-Einstein condensates and its dynamics for non-radial solutions of wave kinetic equations.

The 3-wave kinetic equations play a crucial role in the theory of weak turbulence and have been extensively studied in various contexts. They have been analyzed in \cite{GambaSmithBinh} for stratified ocean flows,  in \cite{soffer2018dynamics} for Bose-Einstein condensates, in \cite{CraciunBinh,GambaSmithBinh,tran2020reaction} for phonon interactions in anharmonic crystal lattices, and in \cite{rumpf2021wave} for beam waves. Additionally, the analysis of four-wave kinetic equations was pioneered by Escobedo and Velazquez in \cite{EscobedoVelazquez:2015:FTB,EscobedoVelazquez:2015:OTT} and has been further explored in several other works \cite{ampatzoglou2024inhomogeneous,collot2024stability,dolce2024convergence,escobedo2024instability,germain2017optimal,germain2024stability,germain2023local,germain2025local,menegaki20222,staffilani2024condensation,staffilani2024energy}. The 6-wave kinetic equations have also been analyzed recently in \cite{pavlovic2025inhomogeneous}.

The rigorous derivation problem of wave kinetic equations has been recently studied in \cite{buckmaster2019onthe,buckmaster2019onset,collot2020derivation,collot2019derivation,deng2022wave,dymov2019formal,dymov2019formal2,dymov2020zakharov,dymov2021large,germain2024universality,grande2024rigorous,hani2023inhomogeneous,hannani2022wave,LukkarinenSpohn:WNS:2011,ma2022almost,staffilani2021wave} and the problem has been completely solved in the  work of Deng and Hani  \cite{deng2023full,deng2021propagation,deng2023long}.
\subsection{Physical context of the model \eqref{3wave}-\eqref{E1}-\eqref{E3}-\eqref{E2}}

The realization of Bose-Einstein condensation (BEC) in trapped atomic vapors of $^{23}$Na \cite{davis1995bose}, $^{87}$Rb \cite{WiemanCornell} and $^7$Li \cite{bradley1995evidence} has started a period of immense theoretical and experimental   research. The experimental results need a theoretical support which is a coupling system between    the coupled non-equilibrium dynamics of both the BEC and the thermal cloud of the trapped Bose gas. 
The spatially homogeneous quantum Boltzmann equation for the density $f(\tau, k)$ of the non-condensate atoms reads
\begin{eqnarray}\label{QB}
	\frac{\partial f}{\partial \tau} 
	\	&= &\ C_{12}[f]\  +  \  C_{22}[f] \ + \ C_{31}[f],\ \ \ f(0, k) \ = \ f_{in}(k),
\end{eqnarray}
where the forms of $C_{12}$, $C_{22}$, $C_{31}$ are given explicitly below
\begin{eqnarray}
	\label{C12Discrete}\nonumber
	&& C_{12}[f](k)  =     \frac{4\pi g^2n_c}{\hbar(2\pi)^3} \int_{\mathbb{R}^d}\int_{\mathbb{R}^d}\int_{\mathbb{R}^d}{\mathrm{d}k_1\mathrm{d}k_2\mathrm{d}k_3}\delta(k_1-k_2-k_3) \\\nonumber
	&&\times (\delta(k-k_1)-\delta(k-k_2)-\delta(k-k_3))\delta(\omega_1-\omega_2-\omega_3)\\
	&& \times K^{12}(k_1,k_2,k_3)  \Big[f_2f_3(f_1+1)-f_1(f_2+1)(f_3+1)\Big],
\end{eqnarray}
\begin{eqnarray}
	\label{C22Discrete}\nonumber
	&& C_{22}[f](k) =     \frac{\pi g^2}{\hbar(2\pi)^6} \int_{\mathbb{R}^d}\int_{\mathbb{R}^d}\int_{\mathbb{R}^d}\int_{\mathbb{R}^d}\mathrm{d}k_1\mathrm{d}k_2\mathrm{d}k_3\mathrm{d}k_4\\\nonumber
	&&\times(\delta(k-k_1)+\delta(k-k_2)-\delta(k-k_3)-\delta(k-k_4))\\\nonumber
	&&\times \delta(\omega_1+\omega_2-\omega_3-\omega_4)\delta(k_1+k_2-k_3-k_4)K^{22}(k_1,k_2,k_3,k_4)\\
	&&\times \Big[f_3f_4(f_2+1)(f_1+1)-f_1f_2(f_3+1)(f_4+1)\Big],
\end{eqnarray}
and
\begin{eqnarray}
	\label{C31Discrete}
	&& C_{31}[f](t,p)  = \frac{3\pi g^2}{\hbar(2\pi)^6} \int_{\mathbb{R}^d}\int_{\mathbb{R}^d}\int_{\mathbb{R}^d}\int_{\mathbb{R}^d}{\mathrm{d}k_1\mathrm{d}k_2\mathrm{d}k_3\mathrm{d}k_4}\\\nonumber
	&&\times(\delta(k-k_1)-\delta(k-k_2)-\delta(k-k_3)-\delta(k-k_4))\\\nonumber
	&&\times \delta(k_1-k_2-k_3-k_4)\delta(\omega_1-\omega_2-\omega_3-\omega_4)K^{31}
	(k_1,k_2,k_3,k_4)\\\nonumber
	&&\times\Big[f_3f_4f_2(f_1+1)-f_1(f_2+1)(f_3+1)(f_4+1)\Big],
\end{eqnarray}
in which $\omega_i$, $f_i$ stand for $\omega(k_i),f(k_i)$, $k\in\mathbb{R}^d$ is the $d$-dimensional non-zero momentum variable.  
The dispersion relation $\omega(k)$ is the Bogoliubov dispersion relation $\Big[\frac{gn_c\hbar^2}{m}p^2  + (\frac{\hbar^2p^2}{2m})^2\Big]^\frac12$, but at very low temperature, can be approximated by the phonon dispersion relation \eqref{E2} (see \cite{reichl2016modern,reichl2019kinetic}). The quantity $\hbar$ is the reduced Planck constant, $g$ is the interaction coupling constant, $n_c$ is the density of the condensate and $m$ is the mass of the particles. The quantities $$K^{12}(k_1,k_2,k_3)= K^{12}(k_2,k_1,k_3) = K^{12}(k_2,k_3,k_1),$$ $$ K^{22}(k_1,k_2,k_3,k_4)=K^{22}(k_2,k_1,k_3,k_4)=K^{22}(k_2,k_3,k_1,k_4)=K^{22}(k_2,k_3,k_4,k_`),$$ $$K^{31}
(k_1,k_2,k_3,k_4)=K^{31}
(k_2,k_1,k_3,k_4)=K^{31}
(k_2,k_3,k_1,k_4)=K^{31}
(k_2,k_3,k_4,k_1)$$ are symmetric, explicit, positive functions that are the kernels of the collision operators. 

Equation \eqref{QB} has a long history. 
In the pioneering work \cite{KD1,KD2,KD3}, Kirkpatrick and Dorfman  began developing a theoretical framework, which introduces a mean-field kinetic equation for the thermal cloud, describing relaxation through ``collisions'' between excitations.
Kirkpatrick-Dorfman's framework was later extended by Zaremba, Nikuni, and Griffin \cite{ZarembaNikuniGriffin:1999:DOT}, who formulated a fully coupled system consisting of a quantum Boltzmann equation for the density function of the thermal cloud and an equation for the BEC wavefunction. The Zaremba-Nikuni-Griffin model has been highly successful in decribing  a wide range of BEC phenomena \cite{GriffinNikuniZaremba:BCG:2009}.  The model was also independently derived by Pomeau, Brachet, Metens, and Rica \cite{PomeauBinh,PomeauBrachetMetensRica}.
In the model, two primary types of collisional processes are considered: 

\begin{itemize}
	\item The $1\leftrightarrow2$ interactions between the condensate and excited atoms, which is described by the collision operator \eqref{C12Discrete}.
	\item The $2\leftrightarrow2$  interactions among the excited atoms themselves, which is described by the collision operator \eqref{C22Discrete}.
\end{itemize}
However, in a later work \cite{ReichlGust:2012:CII}, Reichl and Gust argued that a third, missing collisional process, involving 
$1\leftrightarrow3$  interactions between excitations needs to be taken into account. 
However, a    concise mathematical justification for the existence of the missing collision operator $C_{31}$  was  a challenging open problem for several years until it was  resolved in a recent work by Pomeau and Tran \cite{tran2019boltzmann} (see also the discussion in \cite{tran2021thermal}).
Experimental evidences for this new collision operator have also been done  \cite{reichl2019kinetic}.

However, considering all the $3$ collision operators $C_{12},C_{22},C_{31}$ is a too complicated problem since both operators $C_{22},C_{31}$ are  4-wave kinetic ones while $C_{12}$ is of 3-wave kinetic type. It is therefore a common practice that one   omits the $2$ collision operators $C_{22},C_{31}$ in \eqref{QB} for mathematical purposes (see for instance \cite{cortes2020system,escobedo2023linearized1,escobedo2023linearized,EPV})

\begin{equation}\label{3waveSimplified}
	\begin{aligned} 
		\frac{\partial f}{\partial \tau}(\tau,k) &\ = \ n_c \overline{C}_{12}[ f](\tau,k)\,, \ \ \ \   f(0,k)=f_{in}(k),
	\end{aligned}
\end{equation}
and the collision operator is defined as
\begin{align}\label{C12}
	\begin{split}
		\overline{C}_{12}[f]:=&\int _{\mathbb R^d}\int _{\mathbb R^d} \text{d}k_1\text{d}k_2 K^{12}(k_1,k_2,k_3)\frac1\pi\delta \left(   {\omega (k)}  - {\omega (k_1)} - {\omega (k_2)}  \right)  \delta (k-k_1-k_2)\\
		&\hspace{-.5cm}\times\big[ f(k_1)f(k_2) -(f(k_1)+f(k_2)+1)f(k) \big]\, \\
		&-2\int _{\mathbb R^d}\int _{\mathbb R^d} \text{d}k_1\text{d}k_2 K^{12}(k_1,k_2,k_3)\frac1\pi\delta \left(   {\omega (k_1)}  - {\omega (k)} - {\omega (k_2)}  \right)  \delta (k_1-k-k_2)\\
		&\hspace{-.5cm}\times\big[ f(k)f(k_2) -(f(k)+f(k_2)+1)f(k_1) \big]\,,
	\end{split}
\end{align}
where we have normalized all the constants to be zero. 

Assuming that a lot of the particles are condensed into the BEC and the BEC is quite stable, we can suppose  $n_c$ is a (big) constant, that can be normalized to be $1$ by a time scaling $$\tau\to\frac{\tau}{n_c}.$$

Next, we perform a common simplification strategy to Equation \eqref{3waveSimplified}, which is to keep only the second order terms while omitting first order terms in $\big[ f(k_1)f(k_2) -(f(k_1)+f(k_2)+1)f(k) \big]$, $\big[ f(k)f(k_2) -(f(k)+f(k_2)+1)f(k_1) \big]$ and reduce them  to $\big[ f(k_1)f(k_2) -(f(k_1)+f(k_2))f(k) \big]$ and $\big[ f(k)f(k_2) -(f(k)+f(k_2))f(k_1) \big]$, we get \eqref{3wave}-\eqref{E1}, where $d$ is replaced by $2$ and $K^{12}(k_1,k_2,k_3)$  is replaced by 
\begin{equation}
	\label{Kernel12}K^{12}(k_1,k_2,k_3)\longrightarrow\frac1\pi[|k|+|k_1|+|k_2|]^{-1}.
\end{equation}  Note that the constant $\frac1\pi$ is put so that we simplify  the unnecessary constants in later computations and does not play any significant role. 

{\bf Acknowledgment} We would like to express our gratitude to  J. J.-L. Velazquez,  H. Spohn, J. Lukkarinen for fruitful discussions on the topics. We would like to express our gratitude to  R. Grande and M. Dolce for several useful suggestions to improve the quality of the manuscript.

\section{The Settings and main results}

  \subsection{Measure space}

 Let $\Xi\ge 3$ be a fixed  prime number and
 set 
 \begin{equation}\label{Measure:1}\begin{aligned}
 		\Upsilon_\eta(\mu,\nu) \ = \ & \ \Xi^{-\eta}\Upsilon(\mu,\nu) \ =\ \ \Xi^{-\eta}(\Xi \mu-\nu) 
 		\ \ \ ;
 		\\ \ \  \mu\ =\ & 1, 2, 3,\cdots\ \ ;\ \ \nu=1,\cdots, \Xi-1; \ \ \eta=0,1,2,3,\cdots
 \end{aligned}\end{equation}

 We also introduce   the following   sets  
\begin{equation}\label{Measure:2}\begin{aligned}
		\Theta_{\eta}\ =\ & \left\{ \Upsilon_\eta(\mu,\nu);\ \ \mu=1,2,3,\cdots\ \ ;\ \ \nu=1,\cdots, \Xi-1 \ \right\}
		\ \ ,\ \ \eta=0,1,2,3,\cdots;\\
		\Lambda_{\alpha}\ =\ &\bigcup_{\alpha\geq\eta}\Theta_{\eta}, \ \ \ 	\Theta_{-1}\ =\  \{0\}, \ \ \ \Lambda \ = \ \Lambda_{0}\ \bigcup\ \Theta_{-1}.\end{aligned}
\end{equation}
 We denote
 \begin{equation}\label{Measure:3}\begin{aligned}
 		\mathbb{S}\ =\ & \Big\{ \widehat k\in\mathbb{R}^2 ~~ | ~~|\widehat k|=1\Big\}. \end{aligned}
 \end{equation}
 
 For a vector $k\in\mathbb{R}^2$, we write $k=|k|\widehat k$, with $\widehat k\in\mathbb{S}$. We then can identify $\mathbb{R}^2$ with $  \left[  0,\infty\right)  \times \mathbb{S}  $.

 \begin{definition}\label{Measure:4}
 We call $ \Lambda  \times \mathbb{S}  $ a circular lattice.
 	Let $\mathfrak C>1$, $\gamma\in(0,1]$ be  fixed constants. We denote by $\mathcal M_+([0,\infty))$ the space of positive Radon measure defined on $[0,\infty)$ and we endow $\mathcal M_+([0,\infty))$ with the standard weak topology \cite{brezis2011functional,rudin1974real}.  We  denote as $\mathfrak{S}$ the space of   measures $f(|k|\widehat k)\in \mathcal M_+([0,\infty))\times L^\infty(\mathbb{S})$ defined on  the circular lattice $ (|k|,\widehat k)\in\Lambda_{0}  \times \mathbb{S}  $
 	such that 
 	\begin{itemize}
 		\item[(i)] $\forall f\in  \mathfrak{S}$, $k =|k|\widehat{k}\in\mathbb{R}^2=\left[  0,\infty\right)  \times \mathbb{S} $,  we have
 		\begin{equation}\label{Measure:5}\begin{aligned}
 				f(k) &\ = \ \sum_{\eta=0}^\infty\sum_{\mu=1}^\infty\sum_{\nu=1}^{\Xi-1}\delta_{\{|k|= \Upsilon_\eta(\mu,\nu)\}}f_{\Upsilon_\eta(\mu,\nu)}(\widehat k) \ + \ \delta_{\{|k|=0\}}f_{-1}(\widehat k), \end{aligned}\end{equation}
 		for $f_{\Upsilon_\eta(\mu,\nu)}(\widehat k),f_{-1}(\widehat k)\in L^\infty(\mathbb{S})$. 
		 		\item[(ii)] 	We  set
		 		\begin{equation}\label{Measure:E2:1} f_\eta(|k|\widehat{k}) =\sum_{\mu=1}^\infty\sum_{\nu=1}^{\Xi-1}\delta_{\{|k|= \Upsilon_\eta(\mu,\nu)\}}f_{\Upsilon_\eta(\mu,\nu)}(\widehat k).\end{equation} 
		 		Thus,
 		\begin{equation}\label{Lemma:CollisionDiscrete:E1}
 			f \ =\ \sum_{\eta=0}^{\infty}f_{\eta} \ + \ f_{-1}
 		\end{equation} in which for $\eta\ge0$
 		\begin{equation}\label{Measure:E2}\begin{aligned}
 				\int_{  \left[
 					0,\infty\right)  \setminus \Theta_{\eta }}\mathrm{d}|k|f_\eta(|k|\widehat{k})  \ =\ &0, \end{aligned}\end{equation}
 		for a.e. $\widehat{k}\in \mathbb{S}$. 
		
		 We impose the condition,
 		
 		\begin{equation}\label{Measure:6}
 			\left\Vert f\right\Vert _{\mathfrak{S}}\equiv\max\left\{\sup_{\eta\geq0}\mathfrak C^\eta\sup_{\widehat{k}\in \mathbb{S}}\left(
 			  \int_{\Theta_{\eta} }\mathrm{d}|k|\Big|f_\eta(|k|\widehat{k})\Big|
 			\right),\sup_{\widehat{k}\in \mathbb{S}}\Big|f_{-1}(\widehat k)\Big|\right\}<\infty.
 		\end{equation}

 	\end{itemize}

 \end{definition}

  We  write, for $f_{in}(k)=f(0,k)\in \mathfrak{S}$
  \begin{equation}\label{Initial:1}
  	f_{in}(k) \ =\ \sum_{\eta=0}^{\infty}f_{\eta}(0,k)
  \end{equation} in which 
  \begin{equation}\label{Measure:E2}\begin{aligned}
  		\int_{  \left[
  			0,\infty\right)  \setminus \Theta_{\eta }}\mathrm{d}|k|f_
  		\eta(0,|k|\widehat{k})  \ =\ &0, \end{aligned}\end{equation}
  for a.e. $\widehat{k}\in \mathbb{S}$. The following assumptions are imposed on the initial condition.

\begin{itemize}
	\item  There exists  constants $\mathscr C_1> \mathscr C_2>0$ and $
	\mathscr C_3>0$ such that for $\rho\in\mathbb{Z}, \rho\ge 0$ $$ f_\rho(0,|k|\widehat{k})=\delta_{\{|k|=\Xi^{-\rho}\}}f_{\Xi^{-\rho}}(0,\widehat k)$$ 
	with 
	$$\mathscr C_2\frac{\mathscr C_3^\rho}{(\rho!)^\gamma}	\ \le \  f_{\Xi^{-\rho}}(0,\widehat k)\le\ \mathscr C_1\frac{\mathscr C_3^\rho}{(\rho!)^\gamma},$$
	for a.e. $\widehat{k}\in\mathbb{S}$.
	
	In other words,
		\begin{equation}\label{Assum1}
\mathscr C_2\frac{\mathscr C_3^\rho}{(\rho!)^\gamma}	\ \le \ 	\int_{\mathbb{R}_+}{\rm d}|k|  f_\rho(0,|k|\widehat{k}) \ = \ \int_{\{|k|=\Xi^{-\rho}\}}{\rm d}|k|  f_\rho(0,|k|\widehat{k}) \ \le\ \mathscr C_1\frac{\mathscr C_3^\rho}{(\rho!)^\gamma},
	\end{equation}
for a.e. $\widehat{k}\in\mathbb{S}$.  
 
\item  Moreover, for a.e. $\widehat{k}\in\mathbb{S}$
\begin{equation}\label{Assum2}	\int_{  \{|k|=0\}}\mathrm{d}|k|f_{in}(|k|\widehat{k})  \ =\ 0, \ \ \ \frac{\mathscr C_1  }{\mathscr C_2\mathscr C_3 }>10.	\end{equation}

\end{itemize}   
We have the following basic lemma, whose proof can be found in the Appendix.
\begin{lemma}\label{Lemma:Close} For any positive constant	$R>0$, we define set $$A = \Big\{f ~~~\Big| ~~~ \|f\|_{\mathfrak S}\le R\Big\}.$$
Suppose that $\{f_n\}_{n=0}^\infty$ is a sequence  in $A$ that converges to $f$ in the weak topology of $\mathcal{M}_+([0,\infty))$ as $n$ goes to infinity, for a.e. $\widehat{k}\in \mathbb{S}$, then $f\in A$.
\end{lemma}

 \subsection{Main theorem}
 
\begin{definition}\label{Def:Mild} Performing the time change of variables  $t=\tau|k|$ and $f(t,k)=f(\tau|k|,k)$, we obtain from \eqref{3wave}
	
	\begin{equation}\label{3wavenew}
		\begin{aligned} 
			\frac{\partial f}{\partial t}(t,k) \frac{1}{|k|}&\ = \  Q[ f](t,k)\,, \ \ \ \  \ \ \ \   f(0,k)=f_{in}(k).
		\end{aligned}
	\end{equation}

	We say that $f(t,k)$ is a local mild  solution of \eqref{3wave} and \eqref{3wavenew} with a (non-radial) initial condition $f_{in}(k) \ge0$, $f_{in}\in\mathfrak S$ if $f(t,k)\ge0$ and there exists $T>0$ such that $f(t,k)\in C([0,T), \mathfrak S)$ and 
for all $\phi\in C(\mathbb{R}^2)$, for all $t\in [0,T)$, we have
\begin{equation}\label{3wavemild}
	\begin{aligned}
		\int_{\mathbb{R}^2}\mathrm{d}k f(t,k)\phi(k)\frac{1}{|k|}\ = \ & 	\int_{\mathbb{R}^2}\mathrm{d}k f_{in}(k)\phi(k)\frac{1}{|k|} \  + \ \int_0^t\int_{\mathbb{R}^3}\mathrm{d}k	 Q \left[ f\right]\phi(k).
	\end{aligned}
\end{equation}

	We say that $f(t,k)\in C([0,\infty), \mathfrak S)$ is a global mild  solution of \eqref{3wave} and \eqref{3wavenew} with a (non-radial) initial condition $f_{in}(k) \ge0$, $f_{in}\in\mathfrak S$ if $f(t,k)\ge0$ and
	for all $\phi\in C(\mathbb{R}^2)\cap L^\infty_{|k|}L^1_{\widehat{k}}$ and for all $t\ge 0$, \eqref{3wavemild} holds true. 
\end{definition}

\begin{theorem}\label{maintheorem}
 Under Assumptions \eqref{Assum1}-\eqref{Assum2},  \eqref{3wavenew} (and \eqref{3wave})  has a global mild solution in $f(t,k)\in C([0,\infty), \mathfrak S)$ in the sense of Definition \ref{Def:Mild}. Moreover, for a.e. $\widehat k\in\mathbb S,$ one of  following alternatives holds true.
 
 \begin{itemize}
 	\item{(I)} Finite time condensation: There exists a finite time $0<T(\widehat{k})<\infty$ such that
 \begin{equation}
 	\label{maintheorem:2}
 		\int_{\{0\}}\mathrm d|k|f(T(\widehat{k}),|k|\widehat{k}) \ \ >0,
 		 \end{equation} with $k=|k|\widehat{k}$.
 	\item{(II)} Infinite time condensation:  \begin{equation}\label{maintheorem:1}
 		\lim_{t\rightarrow\infty}\int_{\{|k|=0\}} {\rm d}|k|\,  {f(t,k)}\varphi(|k|)=\varphi(0)\int_{\mathbb{R}_+} {\rm d}|k|{f}_{in}(k),
 	\end{equation}
 	for any test function $\varphi(|k|)\in C([0,\infty))$.

There exist a time sequence $\{\tau_n\}_{n=1}^\infty$, with $\tau_1<\tau_2<\cdots<\tau_n<\cdots$ and $\lim_{n\to\infty}\tau_n=\infty$, and a constant $N_0>1$ such that for all $n>N_0$ and for a.e. $\widehat k\in\mathbb S,$ 
) 
 	\begin{align}\label{maintheorem:2}	\begin{split}
 			&\int_{\{|k|=\Xi^{-n}\}}\mathrm{d}|k|{f(t,k)}  \
 			\ >\ 	 \mathfrak{C}_1 (t+1)\int_{\{|k|\le\Xi^{-n}\}}\mathrm{d}|k|{f(0,k)}, 	\end{split}
 	\end{align}
 	for all $t\in[\tau_{n-1},\tau_n)$, where $\mathfrak C_1>0$ is a   constant independent of $t,\widehat k $ and $n,N_0$. 
 \end{itemize}

\end{theorem}
\begin{remark}\label{R1} 
	In the above theorem, the initial condition  concentrates on the circles $|k|=\Xi^{-\rho}$, $\rho\ge 0$, with the magnitude $\mathcal{O}\Big((\rho!)^{-\gamma}\Big)$ (see Figure \ref{fig1}). That means we construct initial data such that the smaller the radius of the circle is, the smaller the initial is assumed to be. The initial condition decays  on circles near the origin.
	
	\begin{figure}
			\includegraphics[scale=0.3]{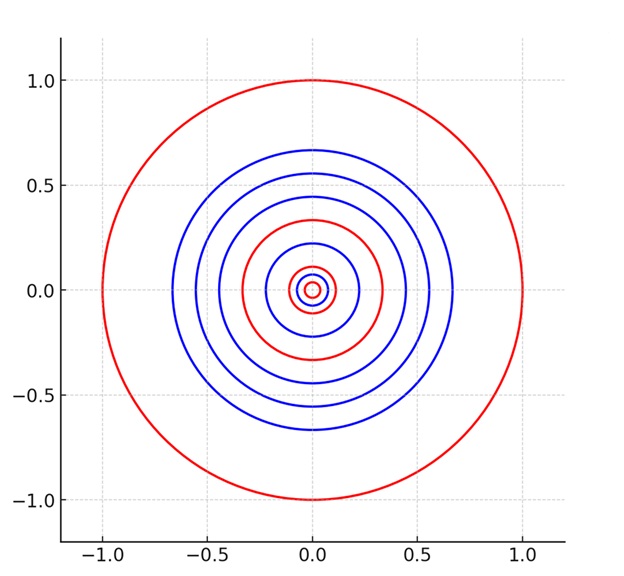}
				\caption{The circular lattice $\Lambda$. The initial conditions are  supported on the red circles.}
		\label{fig1}
	\end{figure}
	
	The theorem shows that the support of the solution remains  on the circular lattice $\Lambda\times\mathfrak S$ as time evolves. Moreover, 
from \eqref{maintheorem:1}, even though for a.e. $\widehat{k}\in\mathbb{S}$
\begin{equation}	\int_{  \{|k|=0\}}\mathrm{d}|k|f_{in}(|k|\widehat{k})   \ =\ 0, 	\end{equation} the solution $f(t,k)$ either forms a  Dirac mass - Bose-Eistein Condensate $\int_{\mathbb{R}_+} {\rm d}|k|{f}_{in}(|k|\widehat k)\delta_{\{|k|=0\}}$ in finite time or converges weakly to the Bose-Eistein Condensate as time goes to infinity for a.e. direction $\widehat k\in\mathbb{S}$. In the latter case, from  \eqref{maintheorem:2}, one can see that the solution accumulates at circles $|k|=\Xi^{-n}$ with a growth rate at least  linearly in time $\mathfrak{C}_1 (t+1)$.

\end{remark}
 
 \section{Estimates for   the collision operator and a local existence result}

 \begin{proposition} \label{Lemma:WeakFormulation}
 	For any suitable test function $\varphi\in  L^\infty_{|k|}L^1_{\widehat{k}}$, the following weak  formulation holds for the collision operator ${Q}$ 
 	\begin{align}\label{Lemma:WeakFormulation:Eq1}
 		\begin{split}
 			\int_{\mathbb{R}^2}& {\rm d}k\,  {Q}[f](k)\varphi(k)=\int_{\mathbb{R}^2}\int_{\mathbb{R}^2}\int_{\mathbb{R}^2}{\rm d}k\,{\rm d}k_1\,{\rm d}k_2\, [|k|+|k_1|+|k_2|]^{-1}\frac1\pi\delta(k-k_1-k_2)\\
 			&\times \delta(|k|-|k_1|-|k_2|)\Big[ f(k_1)f(k_2)-f(k_1)f(k)-f(k_2)f(k)\Big]\\
 			&\times \Big[\varphi(k)-\varphi(k_1)-\varphi(k_2)\Big]\\
 			&=\int_{\mathbb{R}^2}{\rm d}k_1\int_{\mathbb{R}_+}|k_2|{\rm d}|k_2|\,\frac{1}{|k_{1}|}\mathcal{W}\big(|k_1|+|k_2|, |k_1|, |k_2|\big)\Big[ f(k_1)f(|k_2|\widehat{k_1}) \\
 			& -f(k_1)f(k_1+|k_2|\widehat{k_1})-f(|k_2|\widehat{k_1})f(k_1+|k_2|\widehat{k_1}) \Big]\\
 			&\times \Big[\varphi(k_1+|k_2|\widehat{k_1})-\varphi(k_1)-\varphi(|k_2|\widehat{k_1})\Big]\,,
 		\end{split}
 	\end{align}
 where, for $k\in\mathbb{R}^2$, we denote $k=|k|\widehat{k}$, $\widehat{k}\in\mathbb{S}$.

 When $\varphi(k)\in  L^\infty_{|k|}L^1_{\widehat{k}}$ we have
 
 \begin{align}\label{Lemma:WeakFormulation:Eq2}
 	\begin{split}
 		\int_{\mathbb{R}^2}& {\rm d}k\,  {Q}[f](k)\varphi(k)
 		\\
 		\ = \  &2\int_{\mathbb{R}^2}{\rm d}k_1\int_{|k_2|>|k_1|}{\rm d}|k_2|\,\frac{1}{|k_{1}|}  {f(k_1)f(|k_2|\widehat{k_1})} \Big[\varphi(|k_1|\widehat{k_1}+|k_2|\widehat{k_1})\\
 		&\ \ \ \ \  +\varphi(-|k_1|\widehat{k_1}+|k_2|\widehat{k_1})-2\varphi(|k_2|\widehat{k_1})\Big]\,\\
 		&\ \ \ \ \  + \int_{\mathbb{R}^2}{\rm d}k_1\int_{|k_2|=|k_1|}{\rm d}|k_2|\,\frac{1}{|k_{1}|}  {f(k_1)f(|k_2|\widehat{k_1})} \Big[\varphi(2|k_1|\widehat{k_1})-2\varphi(|k_1|\widehat{k_1})+2\varphi(0)\Big]\, .
 	\end{split}
 \end{align} 

In addition, when $\varphi(k)=\tilde{\varphi}(\widehat{k})\chi_{(0,\infty)}(|k|)$ , $\tilde{\varphi}(\widehat{k})\in L^1(\mathbb{S})$,   we have, 
\begin{equation}
	\begin{aligned}\label{Lemma:CollisionDiscrete3}
		\int_{ \mathbb{R}^2}& {\rm d}k\,  {Q}[f](k)\varphi(k)
		\ \le \ 0.
\end{aligned} \end{equation}

 \end{proposition}
 \begin{proof}
  The proof of \eqref{Lemma:WeakFormulation:Eq1} follows a standard argument (see \cite{soffer2020energy}) by performing the change of  variables $k\leftrightarrow k_1$ and $k\leftrightarrow k_2$ 
 		\begin{align}\label{Lemma:WeakFormulation:E4}
 		\begin{split}
 			\int_{\mathbb{R}^2}& {\rm d}k\,  {Q}[f](k)\varphi(k)=\int_{\mathbb{R}^2}\int_{\mathbb{R}^2}\int_{\mathbb{R}^2}{\rm d}k\,{\rm d}k_1\,{\rm d}k_2\, (|k|+|k_1|+|k_2|)^{-1}\frac1\pi\delta(k-k_1-k_2)\\
 			&\times \delta(|k|-|k_1|-|k_2|)\Big[ f(k_1)f(k_2)-f(k_1)f(k)-f(k_2)f(k)\Big]\\
 			&\times \Big[\varphi(k)-\varphi(k_1)-\varphi(k_2)\Big].
 		\end{split}
 	\end{align} Next, we develop
 	\begin{align}\label{Lemma:WeakFormulation:E5}
 		\begin{split}
 			\int_{\mathbb{R}^2}\text{d}k&\,{Q}[f](k)\varphi(k)= \int_{\mathbb{R}^2}\int_{\mathbb{R}^2}(2|k_1|+2|k_2|)^{-1}\frac1\pi\delta(|k_1+k_2|-|k_1|-|k_2|)\\
 			&\times\Big[ f(k_1)f(k_2)-f(k_1)f(k_1+k_2)-f(k_2)f(k_1+k_2)\Big]\\
 			&\times \Big[\varphi(k_1+k_2)-\varphi(k_1)-\varphi(k_2)\Big]\text{d}k_{1}\,\text{d}k_2\,.
 		\end{split}
 	\end{align}
 We compute $$|k_1+k_2|-|k_1|-|k_2|=\big(|k_{1}|^{2} + |k_{2}|^{2} + 2|k_{1}||k_{2}|\widehat{k_1}\cdot\widehat{k_2}\big)^{1/2}-|k_{1}|-|k_{2}|\,,$$
 yielding  $|k_1+k_2|-|k_1|-|k_2|=0$ if and only if $\widehat{k_1}\cdot\widehat{k_2}=1$.  By a polar change of variable,  we obtain, for any continuous function $G(k_2)$
 	\begin{align*}
 		\int_{\mathbb{R}^{2}}&\text{d}k_2\,G(k_{2})\,\delta\big(|k_1+k_2|-|k_1|-|k_2|\big)\\
 		&=\int_{\mathbb{R}_{+}}|k_2| \text{d}|k_2|\int^{2\pi}_{0}\text{d}\phi\int^{1}_{-1}\text{d}s\, G\big(k_{2}(s,\sin(\phi))\big)\delta(y(s))\\
 		&=\int_{\mathbb{R}_{+}}|k_2| \text{d}|k_2|\,\frac{G(|k_{2}|\widehat{k_1})}{y'(1)} = \int_{\mathbb{R}_{+}}|k_2| \text{d}|k_2|\,G(|k_{2}|\widehat{k_1}) \frac{|k_1|+|k_{2}|}{|k_{1}||k_{2}|}\,,
 	\end{align*}   
 	where $y(s)=\big(|k_{1}|^{2} + |k_{2}|^{2} + 2|k_{1}||k_{2}|s\big)^{1/2}-|k_1|-|k_2|$.
 Applying the above identity to \eqref{Lemma:WeakFormulation:E4} proves the second equality in \eqref{Lemma:WeakFormulation:Eq1}.   
 
 We proceed further by  supposing $\varphi(k)=\varphi(|k|\widehat{k})$
 \begin{align}\label{Lemma:WeakFormulation:E6}
 	\begin{split}
 		\int_{\mathbb{R}^2}& {\rm d}k\,  {Q}[f](k)\varphi(k)\ = \ \int_{\mathbb{R}^2}{\rm d}k_1\int_{\mathbb{R}_+}{\rm d}|k_2|\,\frac{1}{|k_{1}|}\Big[ f(k_1)f(|k_2|\widehat{k_1}) \\
 		& -f(k_1)f(k_1+|k_2|\widehat{k_1})-f(|k_2|\widehat{k_1})f(k_1+|k_2|\widehat{k_1})\Big]\\
 		&\times \Big[\varphi(|k_1|\widehat{k_1}+|k_2|\widehat{k_1})-\varphi(|k_1|\widehat{k_1})-\varphi(|k_2|\widehat{k_1})\Big]\,\\
 \ = \  & \int_{\mathbb{R}^2}{\rm d}k_1\int_{\mathbb{R}_+}{\rm d}|k_2|\,\frac{1}{|k_{1}|}  {f(k_1)f(|k_2|\widehat{k_1})}   \Big[\varphi(|k_1|\widehat{k_1}+|k_2|\widehat{k_1})\ -\varphi(|k_1|\widehat{k_1})-\varphi(|k_2|\widehat{k_1})\Big]\,\\
&- 2\int_{\mathbb{R}^2}{\rm d}k_1\int_{\mathbb{R}_+}{\rm d}|k_2|\,\frac{1}{|k_{1}|}{f(k_1)f(k_1+|k_2|\widehat{k_1})}{}  \Big[\varphi(|k_1|\widehat{k_1}+|k_2|\widehat{k_1})\\& -\varphi(|k_1|\widehat{k_1})-\varphi(|k_2|\widehat{k_1})\Big]\,
	\end{split}
 \end{align}
 which gives
  \begin{align}
 	\begin{split}
 \int_{\mathbb{R}^2} {\rm d}k\,  {Q}[f](k)\varphi(k)\ = \ &\int_{\mathbb{R}^2}{\rm d}k_1\int_{\mathbb{R}_+}{\rm d}|k_2|\,\frac{1}{|k_{1}|}  {f(k_1)f(|k_2|\widehat{k_1})}   \Big[\varphi(|k_1|\widehat{k_1}+|k_2|\widehat{k_1})\\
&-\varphi(|k_1|\widehat{k_1})-\varphi(|k_2|\widehat{k_1})\Big]\,\\
&- 2\int_{\mathbb{R}^2}{\rm d}k_1\int_{|k_2|\ge |k_1|}{\rm d}|k_2|\,\frac{1}{|k_{1}|}{f(k_1)f(|k_2|\widehat{k_1})}  \Big[\varphi(|k_2|\widehat{k_1})\\
&-\varphi(|k_1|\widehat{k_1})-\varphi(-|k_1|\widehat{k_1}+|k_2|\widehat{k_1})\Big]\ ,
 	\end{split}
 \end{align}
in which we perform the change of variable $k_1+|k_2| \widehat{k_1} \to k_2$ in the second integral. Next, we split the two integrals and regroup the terms as follows 
\begin{align}\label{Lemma:WeakFormulation:E7}
	\begin{split}
		\int_{\mathbb{R}^2}& {\rm d}k\,  {Q}[f](k)\varphi(k)
		\\
		\ = \  &2\int_{\mathbb{R}^2}{\rm d}k_1\int_{|k_2|>|k_1|}{\rm d}|k_2|\,\frac{1}{|k_{1}|}  {f(k_1)f(|k_2|\widehat{k_1})}  \Big[\varphi(|k_1|\widehat{k_1}+|k_2|\widehat{k_1})\ -\varphi(|k_1|\widehat{k_1})-\varphi(|k_2|\widehat{k_1})\Big]\,\\
		&+ \int_{\mathbb{R}^2}{\rm d}k_1\int_{|k_2|=|k_1|}{\rm d}|k_2|\,\frac{1}{|k_{1}|}  {f(k_1)f(|k_2|\widehat{k_1})}  \Big[\varphi(2|k_1|\widehat{k_1})-2\varphi(|k_1|\widehat{k_1})+2\varphi(0)\Big]\,\\
		&- 2\int_{\mathbb{R}^2}{\rm d}k_1\int_{|k_2|> |k_1|}{\rm d}|k_2|\,\frac{1}{|k_{1}|}{f(k_1)f(|k_2|\widehat{k_1})}  \Big[\varphi(|k_2|\widehat{k_1})\\
		&-\varphi(|k_1|\widehat{k_1})-\varphi(-|k_1|\widehat{k_1}+|k_2|\widehat{k_1})\Big]\ 	\\
			\end{split}
	\end{align}\begin{align*}
	\begin{split}
		\ = \  &2\int_{\mathbb{R}^2}{\rm d}k_1\int_{|k_2|>|k_1|}{\rm d}|k_2|\,\frac{1}{|k_{1}|}  {f(k_1)f(|k_2|\widehat{k_1})}  \Big[\varphi(|k_1|\widehat{k_1}+|k_2|\widehat{k_1})\\
		&-2\varphi(|k_2|\widehat{k_1})+\varphi(-|k_1|\widehat{k_1}+|k_2|\widehat{k_1})\Big]\,\\
		&+ \int_{\mathbb{R}^2}{\rm d}k_1\int_{|k_2|=|k_1|}{\rm d}|k_2|\,\frac{1}{|k_{1}|}  {f(k_1)f(|k_2|\widehat{k_1})} \Big[\varphi(2|k_1|\widehat{k_1})-2\varphi(|k_1|\widehat{k_1})+2\varphi(0)\Big]\, ,
	\end{split}
\end{align*}
yielding \eqref{Lemma:WeakFormulation:Eq2}.

Finally, \eqref{Lemma:CollisionDiscrete3}  follows  from   straightforward computations.
\end{proof}

 \begin{proposition} \label{Lemma:WeakFormulationb}

	For all $\varphi(k)\in   C(\mathbb{R}^2)\cap L^\infty_{|k|}L^1_{\widehat{k}}$, we have \begin{equation}
		\begin{aligned}\label{Lemma:CollisionDiscrete1}
			\int_{ \left(\left[
				0,\infty\right)  \setminus \left(\bigcup_{\eta=0}^{\infty}\Theta_{\eta }\bigcup\{0\}\right)\right)\times \mathbb{S}}& {\rm d}k\,  {Q}[f](k)\varphi(k)
			\ = \ 0,
	\end{aligned} \end{equation}
	for $f\in \mathfrak{S}$.
	
	We have the estimate 
	\begin{equation}\begin{aligned}\label{Lemma:CollisionDiscrete2}
			\Big\|   |{Q}[f](k)||k| \Big\|_{\mathfrak{S}}\
			\le\		 &	 C_{Q}\|f\|_{\mathfrak{S}}^2,
		\end{aligned}\end{equation}
	for some constant $C_{Q}>0$ independent of $f$.
 Moreover, the equation  \eqref{3wavenew} (and \eqref{3wave}) has   a local mild solution in $C([0,T),\mathfrak S)$, for some $T>0$, in the sense of Definition \ref{Def:Mild}.

\end{proposition}
\begin{proof}

First, we  prove \eqref{Lemma:CollisionDiscrete1}.	 For     $M\in\mathbb{N}$, we set 
\begin{equation}\label{Lemma:CollisionDiscrete:E3}
	f^M \ =\ \sum_{\eta=-1}^{M}f_{\eta},
\end{equation}
where we have used \eqref{Measure:E2:1} and \eqref{Lemma:WeakFormulation:Eq2}, and 
\begin{equation} \begin{aligned}\label{Lemma:CollisionDiscrete:E4}
		\int_{\mathbb{R}^2}& {\rm d}k\,  {Q}_M[f](k)\varphi(k)
		\\
		\ = \  &2\int_{\mathbb{R}^2}{\rm d}k_1\int_{|k_2|>|k_1|}{\rm d}|k_2|\, f^M(k_1)f^M(|k_2|\widehat{k_1}) \frac{1}{|k_{1}|}  \Big[\varphi(|k_1|\widehat{k_1}+|k_2|\widehat{k_1})\\
		&-2|k_2|\varphi(|k_2|\widehat{k_1})+\varphi(-|k_1|\widehat{k_1}+|k_2|\widehat{k_1})\Big]\,\\
		&+  \int_{\mathbb{R}^2}{\rm d}k_1\int_{|k_2|=|k_1|}{\rm d}|k_2|\,\frac{1}{|k_1|}  f^M(k_1)f^M(|k_2|\widehat{k_1})  \Big[\varphi(2|k_1|\widehat{k_1})-2\varphi(|k_1|\widehat{k_1})+2\varphi(0)\Big]\, 	\\
		\ = \  &\sum_{\eta,\xi=-1}^M2\int_{\mathbb{R}^2}{\rm d}k_1\int_{|k_2|>|k_1|}{\rm d}|k_2|\,\frac{1}{|k_{1}|}  f_\eta(k_1)f_\xi(|k_2|\widehat{k_1}) \frac{1}{|k_{1}||k_2|}  \Big[\varphi(|k_1|\widehat{k_1}+|k_2|\widehat{k_1})\\
		&-2\varphi(|k_2|\widehat{k_1})+\varphi(-|k_1|\widehat{k_1}+|k_2|\widehat{k_1})\Big]\,\\
		&+\sum_{\eta,\xi=-1}^M  \int_{\mathbb{R}^2}{\rm d}k_1\int_{|k_2|=|k_1|}{\rm d}|k_2|\, f_\eta(k_1)f_\xi(|k_2|\widehat{k_1}) \frac{1}{|k_{1}|}  \Big[\varphi(2|k_1|\widehat{k_1})-2\varphi(|k_1|\widehat{k_1})+2\varphi(0)\Big]\, ,
\end{aligned} 	\end{equation}
for $\varphi(k)\in C(\mathbb{R}^2)\cap L^\infty_{|k|}L^1_{\widehat{k}}$.

We can bound
\begin{equation} \begin{aligned}\label{Lemma:CollisionDiscrete:E6}
		&	\Big|\int_{\mathbb{R}^2} {\rm d}k\,  \Big[{Q}[f](k)-{Q}_M[f](k)\Big]\varphi(k)\Big|\\
		\ \lesssim \  &\|\varphi\|_{L^\infty_{|k|}L^1_{\widehat{k}}}\left(  \sup_{\widehat{k}\in \mathbb{S}}\int_{\bigcup_{\eta\geq
				M+1}\Theta_{\eta}}\mathrm{d}|k|f(|k|\widehat{k})\right) \left(  \sup_{\widehat{k}\in \mathbb{S}}\int_{\bigcup_{\eta\geq-1}\Theta_{\eta}}\mathrm{d}|k|f(|k|\widehat{k})|\right),
\end{aligned} \end{equation}
yielding \begin{equation}\begin{aligned}\label{Lemma:CollisionDiscrete:E7}
		\lim_{M\to\infty}	\Big|\int_{\mathbb{R}^2} {\rm d}k\,  \Big[{Q}[f](k)-{Q}_M[f](k)\Big]\varphi(k)\Big|
		\ = \  &0.
\end{aligned} \end{equation}
Therefore, we can write 

\begin{equation}\begin{aligned}\label{Lemma:CollisionDiscrete:E9}
		&		\int_{\left[
			0,\infty\right)  \times\mathbb{S}}   {\rm d}k\,  {Q}[f](k)\varphi(k)
		\\
		\ = \  &\sum_{\eta,\xi=-1}^\infty2\int_{\mathbb{R}^2}{\rm d}k_1\int_{|k_2|>|k_1|}{\rm d}|k_2|\,  f_\eta(k_1)f_\xi(|k_2|\widehat{k_1}) \frac{1}{|k_1|}  \Big[\varphi(|k_1|\widehat{k_1}+|k_2|\widehat{k_1})\\
		&-2\varphi(|k_2|\widehat{k_1})+\varphi(-|k_1|\widehat{k_1}+|k_2|\widehat{k_1})\Big]\,\\
		&+\sum_{\eta,\xi=-1}^\infty \int_{\mathbb{R}^2}{\rm d}k_1\int_{|k_2|=|k_1|}{\rm d}|k_2|\,  f_\eta(k_1)f_\xi(|k_2|\widehat{k_1}) \frac{1}{|k_1|}  \Big[\varphi(2|k_1|\widehat{k_1})-2\varphi(|k_1|\widehat{k_1})+2\varphi(0)\Big]\,.
\end{aligned}\end{equation} 

By the same argument,  similarly, we also obtain

\begin{equation}\begin{aligned}\label{Lemma:CollisionDiscrete:E9a}
		&		\int_{\left[
			0,\infty\right)  \times\mathbb{S}}   {\rm d}k\,  \Big|{Q}[f](k)\varphi(k)\Big|
		\\
		\ \le \  &\sum_{\eta,\xi=-1}^\infty2\int_{\mathbb{R}^2}{\rm d}k_1\int_{|k_2|>|k_1|}{\rm d}|k_2|\,  f_\eta(k_1)f_\xi(|k_2|\widehat{k_1}) \frac{1}{|k_1|}  \Big[\Big|\varphi(|k_1|\widehat{k_1}+|k_2|\widehat{k_1})\Big|\\
		&+2\Big|\varphi(|k_2|\widehat{k_1})\Big|+2\Big|\varphi(|k_1|\widehat{k_1})\Big|+\Big|\varphi(-|k_1|\widehat{k_1}+|k_2|\widehat{k_1})\Big|\Big]\,\\
		&+\sum_{\eta,\xi=-1}^\infty \int_{\mathbb{R}^2}{\rm d}k_1\int_{|k_2|=|k_1|}{\rm d}|k_2|\, f_\eta(k_1)f_\xi(|k_2|\widehat{k_1}) \frac{1}{|k_1|}  \Big[\Big|\varphi(2|k_1|\widehat{k_1})\Big|+4\Big|\varphi(|k_1|\widehat{k_1})\Big|+2\Big|\varphi(0)\Big|\Big]\,.
\end{aligned}\end{equation}

Next, we will prove that
\begin{equation} \begin{aligned}\label{Lemma:CollisionDiscrete:E10}
		\int_{ \left(\left[
			0,\infty\right)  \setminus \left(\bigcup_{\eta=0}^{\infty}\Theta_{\eta }\bigcup\{0\}\right)\right)\times \mathbb{S}}& {\rm d}k\,  {Q}[f](k)\varphi(k)
		\ = \ 0.
\end{aligned} \end{equation}
To this end, we estimate for a sufficiently large number $M$ and a sufficiently small constant $\varepsilon>0$ 
\begin{equation}\begin{aligned}\label{Lemma:CollisionDiscrete:E11}
		&	\left|	\int_{\left(\left[
			0,\infty\right)  \setminus \left(\bigcup_{\eta=0}^{\infty}\Theta_{\eta }\bigcup\{0\}\right)\right)\times \mathbb{S}}   {\rm d}k\,   {Q}[f](k)\varphi(k)
		\right|	 
		\ \le \ 	\int_{\left(\left[
			0,\infty\right)  \setminus\left(\Lambda_M\bigcup\{0\}\right)\right)\times \mathbb{S}}   {\rm d}k\,  \Big| {Q} [f](k)\varphi(k)\Big| 
			\\ \le \ &\left|	\int_{\left(\left[
			0,\infty\right)  \setminus\bigcup
			_{h\in\left(\Lambda_M\bigcup\{0\}\right)}\left(  h-\frac{\varepsilon}{\Xi^{2M}},h+\frac{\varepsilon
			}{\Xi^{2M}}\right) \right)\times \mathbb{S} }   {\rm d}k\,  \Big|  {Q}[f](k)\varphi(k)\Big| 
		\right|	\\
		&+ \ \left|	\int_{\left(\bigcup
			_{h\in\left(\Lambda_M\bigcup\{0\}\right)}\left(  h-\frac{\varepsilon}{\Xi^{2M}},h\right)\bigcup
			_{h\in\Lambda_M}\left(h,h+\frac{\varepsilon
			}{\Xi^{2M}}\right) \right)\times \mathbb{S} }   {\rm d}k\, \Big|  {Q}[f](k)\varphi(k)\Big| 
		\right|\ 	=: \ A +B.
\end{aligned}\end{equation} 

We first estimate the  first term on the right hand side, $A$. Let $\phi^M_\varepsilon(|k|)$ be  a bounded nonnegative and
continuous test function   being   $1$
in the set $\mathbb{R}_{+}\setminus\left[h-\frac{\varepsilon}{\Xi^{2M}}%
,h+\frac{\varepsilon}{\Xi^{2M}}\right]  $ and vanishes in  small neighbourhoods of all 
$h\in\Lambda_M\bigcup\{0\}.$ We set $\tilde{\varphi}^M_\varepsilon=\phi^M_\varepsilon\varphi$ and  bound

\begin{equation}\begin{aligned}\label{Lemma:CollisionDiscrete:E12}
		A \ \lesssim \ 	  &\Big|\sum_{\max\{\eta,\xi\}\le M}\int_{\mathbb{R}^2}{\rm d}k_1\int_{|k_2|>|k_1|}{\rm d}|k_2|\,\frac{1}{|k_1|}  f_\eta(k_1)f_\xi(|k_2|\widehat{k_1})   \\
		&\times\Big[\Big|\tilde{\varphi}^M_\varepsilon(|k_1|\widehat{k_1}+|k_2|\widehat{k_1})\Big|+\Big|\tilde{\varphi}^M_\varepsilon(|k_2|\widehat{k_1})\Big|+\Big|\tilde{\varphi}^M_\varepsilon(|k_1|\widehat{k_1})\Big|+\Big|\tilde{\varphi}^M_\varepsilon(-|k_1|\widehat{k_1}+|k_2|\widehat{k_1})\Big|\Big]\,\\
		&+ \int_{\mathbb{R}^2}{\rm d}k_1\int_{|k_2|=|k_1|}{\rm d}|k_2|\,  f_\eta(k_1)f_\xi(|k_2|\widehat{k_1})  \frac{1}{|k_1|}\\
		&\times \Big[\Big|\tilde{\varphi}^M_\varepsilon(2|k_1|\widehat{k_1})\Big|+\Big|\tilde{\varphi}^M_\varepsilon(|k_1|\widehat{k_1})\Big|\Big]+\Big|\tilde{\varphi}^M_\varepsilon(0)\Big|\Big]\,\Big|\\
		&+\ \Big|\sum_{\max\{\eta,\xi\}\ge M+1}\int_{\mathbb{R}^2}{\rm d}k_1\int_{|k_2|>|k_1|}{\rm d}|k_2|\,\frac{1}{|k_1|}  f_\eta(k_1)f_\xi(|k_2|\widehat{k_1})  \\
		&\times\Big[\Big|\tilde{\varphi}^M_\varepsilon(|k_1|\widehat{k_1}+|k_2|\widehat{k_1})\Big|+\Big|\tilde{\varphi}^M_\varepsilon(|k_2|\widehat{k_1})\Big|+\Big|\tilde{\varphi}^M_\varepsilon(|k_1|\widehat{k_1})\Big|+\Big|\tilde{\varphi}^M_\varepsilon(-|k_1|\widehat{k_1}+|k_2|\widehat{k_1})\Big|\Big]\,\\
		&+ \int_{\mathbb{R}^2}{\rm d}k_1\int_{|k_2|=|k_1|}{\rm d}|k_2|\, f_\eta(k_1)f_\xi(|k_2|\widehat{k_1})\frac{1}{|k_1|}  \\
		&\times \Big[\Big|\tilde{\varphi}^M_\varepsilon(2|k_1|)\Big|+\Big|\tilde{\varphi}^M_\varepsilon(|k_1|\widehat{k_1})\Big|+\Big|\tilde{\varphi}^M_\varepsilon(0)\Big|\Big]\,\Big|\ = :\  \ A_1 \ + \ A_2.
\end{aligned}\end{equation}

We now study the first term on the right hand side $A_1$. By the  choice of the
function $\phi^M_\varepsilon,$   in order for $ f_\eta(k_1)f_\xi(|k_2|\widehat{k_1}) \phi^M_\varepsilon(|k_1|),$ or $  f_\eta(k_1)f_\xi(|k_2|\widehat{k_1}) \phi^M_\varepsilon(|k_2|)$ in the above expression not to vanish, $|k_1|$ or $|k_2|$ has to be in $\bigcup_{\alpha
	=M+1}^{\infty}\Theta_{\alpha}.$ However $|k_1|\in   \Theta_{\eta},$ $|k_2|\in   \Theta_{\xi},$ by the definition of $f_\eta$ and $f_\xi$. Therefore, both $ f_\eta(k_1)f_\xi(|k_2|\widehat{k_1}) \phi^M_\varepsilon(|k_1|)$ and $  f_\eta(k_1)f_\xi(|k_2|\widehat{k_1}) \phi^M_\varepsilon(|k_2|)$ vanish. By the  choice of the
function $\phi^M_\varepsilon,$   in order for $ f_\eta(k_1)f_\xi(|k_2|\widehat{k_1})\phi^M_\varepsilon(|k_1|+|k_2|)$ not to vanish, $|k_1|+|k_2|   \in\bigcup_{\rho
	=M+1}^{\infty}\Theta_{\rho},$ while $|k_1|\in   \Theta_{\eta},$ $|k_2|\in   \Theta_{\xi},$ by the definition of $f_\eta$ and $f_\xi$. It is clear that $|k_1||k_2|\ne 0$. We suppose $|k_1|=\Upsilon_\eta(\mu,\nu)$, $|k_2|=\Upsilon_\xi(\mu',\nu')$, $|k_1|+|k_2|=\Upsilon_\rho(\mu'',\nu'')$, that leads to
$$\Upsilon_\eta(\mu,\nu)\ +\ \Upsilon_\xi(\mu',\nu') \ = \ \Upsilon_\rho(\mu'',\nu''),$$
yielding
$$\Xi^{-\eta}\Upsilon(\mu,\nu)\ +\ \Xi^{-\xi}\Upsilon(\mu',\nu') \ = \ \Xi^{-\rho}\Upsilon(\mu'',\nu''),$$
then
$$\Xi^{\rho-\eta}\Upsilon(\mu,\nu)\ +\ \Xi^{\rho-\xi}\Upsilon(\mu',\nu') \ = \ \Upsilon(\mu'',\nu'').$$
Since $\rho>M\ge \max\{\xi,\eta\}$, the left hand side of the above equation is divisible by 
$\Xi$, while the right hand side is not. As thus, $ f_\eta(k_1)f_\xi(|k_2|\widehat{k_1})\phi^M_\varepsilon(|k_1|+|k_2|)$ vanishes.  By the  choice of the
function $\phi^M_\varepsilon,$   in order for $ f_\eta(k_1)f_\xi(|k_2|\widehat{k_1})\phi^M_\varepsilon(-|k_1|+|k_2|)$ not to vanish, $-|k_1|+|k_2|   \in\bigcup_{\rho
	=M+1}^{\infty}\Theta_{\alpha},$ while $|k_1|\in   \Theta_{\eta},$ $|k_2|\in   \Theta_{\xi},$ by the definition of $f_\eta$ and $f_\xi$. Arguing similarly as above, we also deduce that $ f_\eta(k_1)f_\xi(|k_2|\widehat{k_1})\phi^M_\varepsilon(-|k_1|+|k_2|)$ vanishes. As a result, $A_1=0$.

Next we will establish a bound for $A_2$. Similar with \eqref{Lemma:CollisionDiscrete:E6}, we bound 
\begin{equation} \begin{aligned}\label{Lemma:CollisionDiscrete:E13}
		\Big|A_2\Big|
		\ \lesssim \  &\|\varphi\|_{L^\infty_{|k|}L^1_{\widehat{k}}}\left(  \sup_{\widehat{k}\in \mathbb{S}}\int_{\bigcup_{\eta\geq
				0}\Theta_{\eta}\bigcup\{0\}}\mathrm{d}|k|f(|k|,\widehat{k})\right)\left(  \sup_{\widehat{k}\in \mathbb{S}}\int_{\bigcup_{\eta\geq
				M}\Theta_{\eta}}\mathrm{d}|k|f(|k|,\widehat{k})\right)
\end{aligned} \end{equation} 
$$\longrightarrow 0 \mbox{ as } M\to\infty.$$
Therefore, we find  \begin{equation} \label{Lemma:CollisionDiscrete:E14}
	\Big|A\Big|
	\longrightarrow 0 \mbox{ as } M\to\infty.
\end{equation}

Let $\varpi>0$ be a small constant satisfying ${\ln\left(  \frac
	{1}{\varpi}\right)  }>M{\ln\left(  \Xi\right)  }$. We define a
nonnegative continuous test function  $\varphi_{\varpi}$ satisfying  $\varphi_{\varpi}=1\ $in $\left[  h-\frac{\varpi}%
{2},h+\frac{\varpi}{2}\right]  ,\ \varphi_{\varpi}=0\ $in $\mathbb{R}_{+}%
\setminus\left(  h-\varpi,h+\varpi\right)$ with $h\in\Lambda_M\bigcup\{0\}$. We set $\tilde{\varphi}_\varpi=\varphi_{\varpi}\varphi$ and  bound \begin{equation}\begin{aligned}\label{Lemma:CollisionDiscrete:E15}
		B \ \lesssim \ 	  &\Big|\sum_{\eta,\xi=-1 }^\infty\int_{\mathbb{R}^2}{\rm d}k_1\int_{|k_2|>|k_1|}{\rm d}|k_2|\,  f_\eta(k_1)f_\xi(|k_2|\widehat{k_1}) \frac{1}{|k_1|} \\
				&\times\Big[\Big|\tilde{\varphi}_\varpi(|k_1|\widehat{k_1}+|k_2|\widehat{k_1})\Big|+\Big|\tilde{\varphi}_\varpi(|k_2|\widehat{k_1})\Big|+\Big|\tilde{\varphi}_\varpi(|k_1|\widehat{k_1})\Big|+\Big|\tilde{\varphi}_\varpi(-|k_1|\widehat{k_1}+|k_2|\widehat{k_1})\Big|\Big]\,\\
		& + \int_{\mathbb{R}^2}{\rm d}k_1\int_{|k_2|=|k_1|}{\rm d}|k_2|\,  f_\eta(k_1)f_\xi(|k_2|\widehat{k_1}) \frac{1}{|k_1|} \\
		&\times \Big[\Big|\tilde{\varphi}_\varpi(2|k_1|\widehat{k_1})\Big|+\Big|\tilde{\varphi}_\varpi(|k_1|\widehat{k_1})\Big|+\Big|\tilde{\varphi}_\varpi(0)\Big|\Big]\,\Big|.
\end{aligned}\end{equation} 
We first prove that the first  term in \eqref{Lemma:CollisionDiscrete:E15} tends to $0$ as $\varpi\to 0$. Suppose the contrary, we bound this term as follows
\begin{equation}\begin{aligned}\label{Lemma:CollisionDiscrete:E16}
		&\sum_{\eta,\xi=-1}^\infty\int_{\mathbb{R}^2}{\rm d}k_1\int_{|k_2|>|k_1|}{\rm d}|k_2|\,  f_\eta(k_1)f_\xi(|k_2|\widehat{k_1})  \frac{1}{|k_1|}  \Big|\tilde{\varphi}_\varpi(|k_1|\widehat{k_1}+|k_2|\widehat{k_1})\Big|\\
		\lesssim \  & \sum_{h\in\Lambda_M}\sum_{\eta,\xi=-1}^\infty\|\varphi\|_{L^\infty_{|k|}L^1_{\widehat{k}}} \sup_{\widehat{k}\in \mathbb{S}^2} \iint_{0<||k_1|+|k_2|-h|<\varpi,|k_2|>|k_1|}{\rm d}|k_1|{\rm d}|k_2|\, f_\eta(k_1)f_\xi(|k_2|\widehat{k_1}).
\end{aligned}\end{equation} 
We first consider the case $h\in\Lambda_M$, and $|k_1||k_2|\ne0$. From the above equation, we set $|k_1|=\Upsilon_\eta(\mu,\nu)$ ,$|k_2|=\Upsilon_\xi(\mu',\nu')$, $h=\Upsilon_\rho(\mu'',\nu'')$, and obtain
$$0<|\Upsilon_\eta(\mu,\nu)\ +\ \Upsilon_\xi(\mu',\nu') \ - \ \Upsilon_\rho(\mu'',\nu'')|<\varpi,$$
yielding
$$0<|\Xi^{-\eta}\Upsilon(\mu,\nu)\ +\ \Xi^{-\xi}\Upsilon(\mu',\nu') \ - \ \Xi^{-\rho}\Upsilon(\mu'',\nu'')|<\varpi.$$

We denote  as $\delta=\max\left\{\eta,\xi
, M\right\}\ge\rho$.    Then 
$$0<|\Xi^{\delta-\eta}\Upsilon(\mu,\nu)\ +\ \Xi^{\delta-\xi}\Upsilon(\mu',\nu') \ - \ \Xi^{\delta-\rho}\Upsilon(\mu'',\nu'')|<\varpi\Xi^{\delta}.$$

Since  $|\Xi^{\delta-\eta}\Upsilon(\mu,\nu)\ +\ \Xi^{\delta-\xi}\Upsilon(\mu',\nu') \ - \ \Xi^{\delta-\rho}\Upsilon(\mu'',\nu'')|$ is an integer,  we bound   	$1<\varpi\Xi^{\delta}.$ 
We deduce 
\[
{\ln\left(  \Xi\right)  }\delta\  = \ {\ln\left(  \Xi\right)  }\max\left\{\eta,\xi
,M\right\}  \geq{\ln\left(  
	\frac{1}{\varpi}\right)  },
\]
contradicting the fact that ${\ln\left(  \frac
	{1}{\varpi}\right)  }>M{\ln\left(  \Xi\right)  }$. We now consider the case $h\in\Lambda_M$, and $|k_1|=0\ne |k_2|$. Setting $|k_2|=\Upsilon_\xi(\mu',\nu')$, $h=\Upsilon_\rho(\mu'',\nu'')$, we find $$0<|\Xi^{-\xi}\Upsilon(\mu',\nu') \ - \ \Xi^{-\rho}\Upsilon(\mu'',\nu'')|<\varpi.$$
Denoting  as $\delta=\max\left\{\xi
, M\right\}\ge\rho$, we find   
$$0<| \Xi^{\delta-\xi}\Upsilon(\mu',\nu') \ - \ \Xi^{\delta-\rho}\Upsilon(\mu'',\nu'')|<\varpi\Xi^{\delta},$$
yielding   	$1<\varpi\Xi^{\delta}.$ 
This contradicts the fact that ${\ln\left(  \frac
	{1}{\varpi}\right)  }>M{\ln\left(  \Xi\right)  }$. Similarly, when $|k_1|=|k_2|=0,$ we also have a contradiction. 

Next, we consider the case when $h=0$. If  $0\ne |k_1||k_2|$ we set $|k_1|=\Upsilon_\eta(\mu,\nu)$ ,$|k_2|=\Upsilon_\xi(\mu',\nu')$, and obtain
$0<|\Upsilon_\eta(\mu,\nu)\ +\ \Upsilon_\xi(\mu',\nu')|<\varpi.$
We set   $\delta=\max\left\{\eta,\xi
, M\right\}
$ and also find $1<\varpi\Xi^{\delta},$ leading to another contradiction.  Suppose that  $|k_1|=0\ne|k_2|$, we set $|k_2|=\Upsilon_\xi(\mu',\nu')$, and obtain
$0<|\Upsilon_\xi(\mu',\nu')|<\varpi.$ We set   $\delta=\max\left\{\xi
, M\right\}
$ and also obtain $1<\varpi\Xi^{\delta},$ which is also a contradiction. Similarly, when $|k_1|=|k_2|=0,$ we also have a contradiction. Therefore, the first  term in \eqref{Lemma:CollisionDiscrete:E15} tends to $0$ as $\varpi\to 0$.

We can apply the same strategy to the other terms in $B$, and finally get
\begin{equation}\begin{aligned}\label{Lemma:CollisionDiscrete:E18}
		B & \longrightarrow 0\mbox{ as } \varpi \to 0.
\end{aligned}\end{equation} 
Therefore   \eqref{Lemma:CollisionDiscrete1} is proved and we only need to  prove \eqref{Lemma:CollisionDiscrete2}.
Putting together \eqref{Lemma:CollisionDiscrete:E14} and \eqref{Lemma:CollisionDiscrete:E18}, we obtain
\eqref{Lemma:CollisionDiscrete:E10}, which immediately leads to

\begin{equation} \begin{aligned}\label{Lemma:CollisionDiscrete:E19}
		&	\int_{ \left[
			0,\infty\right)  \times \mathbb{S}}& {\rm d}k\,  {Q}[f](k)\varphi(|k|\widehat{k})
		\ = \  \int_{ \left( \bigcup_{\rho=0}^{\infty}\Theta_{\rho }\bigcup\{0\}\right)\times \mathbb{S}}& {\rm d}k\,  {Q}[f](k)\varphi(|k|\widehat{k}).
\end{aligned} \end{equation}

We bound for $\rho\ge0$, by similar arguments used to obtain \eqref{Lemma:CollisionDiscrete:E6}
\begin{equation}\begin{aligned}\label{Lemma:CollisionDiscrete:E20}
	&	\sup_{\widehat{k}\in 
			\mathbb{S}}	\mathfrak C^\rho\int_{ \Theta_{\rho }} {\rm d}|k|\,|k|   \Big|{Q}[f](k)\Big|\\
		\lesssim\		 & \mathfrak C^\rho\sum_{h\in\Theta_{\rho }}\sum_{\eta,\xi=-1}^\infty\sup_{\widehat{k}\in \mathbb{S}}\iint_{|k_1|+|k_2|=|h|,|k_2|\ge|k_1|}{\rm d}|k_1|{\rm d}|k_2|\, f_\eta(k_1)f_\xi(|k_2|\widehat{k_1}) \\
		&+ \mathfrak C^\rho\sum_{h\in\Theta_{\rho }}\sum_{\eta,\xi=-1}^\infty\sup_{\widehat{k}\in\mathbb{S}} \iint_{-|k_1|+|k_2|=|h|,|k_2|\ge|k_1|}{\rm d}|k_1|{\rm d}|k_2|\,  f_\eta(k_1)f_\xi(|k_2|\widehat{k_1}) \\
		&+ \mathfrak C^\rho\sum_{h\in\Theta_{\rho }}\sum_{\eta,\xi=-1}^\infty \sup_{\widehat{k}\in \mathbb{S}} \iint_{|k_2|=|h|,|k_2|\ge|k_1|}{\rm d}|k_1|{\rm d}|k_2|\,  f_\eta(k_1)f_\xi(|k_2|\widehat{k_1}) \\
		&+ \mathfrak C^\rho\sum_{h\in\Theta_{\rho }}\sum_{\eta,\xi=-1}^\infty \sup_{\widehat{k}\in 
			\mathbb{S}}\iint_{|k_1|=|h|,|k_2|\ge|k_1|}{\rm d}|k_1|{\rm d}|k_2|\, f_\eta(k_1)f_\xi(|k_2|\widehat{k_1}) \\
		\lesssim\		 &	\mathfrak C^\rho\left(  \sup_{\widehat{k}\in \mathbb{S}}\int_{\bigcup_{\eta\geq
				\rho}\Theta_{\eta}}\mathrm{d}|k|f(|k|,\widehat{k})\right)\left(  \sup_{\widehat{k}\in \mathbb{S}}\int_{\bigcup_{\eta\geq
				-1}\Theta_{\eta}}\mathrm{d}|k|f(|k|,\widehat{k})\right)\\
		\lesssim\		 &	  \|f\|_{\mathfrak{S}}^2,
\end{aligned}\end{equation} 
yielding
\begin{equation}\begin{aligned}\label{Lemma:CollisionDiscrete:E21}
\sup_{\widehat{k}\in 
	\mathbb{S}}	\mathfrak C^\rho\int_{ \Theta_{\rho }} {\rm d}|k|\,|k|  \Big|{Q}[f](k)\Big|\
		\lesssim\		 &	\|f\|_{\mathfrak{S}}^2.
\end{aligned}\end{equation} 
Similarly
\begin{equation}\begin{aligned}\label{Lemma:CollisionDiscrete:E21}
		\sup_{\widehat{k}\in 
			\mathbb{S}}	\int_{ \{0\}} {\rm d}|k|\,|k|  \Big|{Q}[f](k)\Big|\
		\lesssim\		 &	\|f\|_{\mathfrak{S}}^2.
\end{aligned}\end{equation} 
We then deduce
\begin{equation}\begin{aligned}\label{Lemma:CollisionDiscrete:E22}
	\max\left\{	\sup_{\widehat{k}\in 
		\mathbb{S}}	\int_{ \{0\}} {\rm d}|k|\,|k|  \Big|{Q}[f](k)\Big|,	\sup_{\rho\ge0}\sup_{\widehat{k}\in 
			\mathbb{S}}	\Big|\int_{ \Theta_{\rho } } {\rm d}|k|\,|k| \mathfrak C^\rho |{Q}[f](k)|\Big|\
		\right\}\lesssim\		 &	 \|f\|_{\mathfrak{S}}^2,
\end{aligned}\end{equation} 
which implies
\begin{equation}\begin{aligned}\label{Lemma:CollisionDiscrete:E25}
		\Big\|  |k| |{Q}[f](k)| \Big\|_{\mathfrak{S}}\
		\le\		 &	 C_{Q}\|f\|_{\mathfrak{S}}^2,
\end{aligned}\end{equation}
for some constant $C_{Q}>0$ independent of $f$.

For $g\in C([0,T),\mathfrak S)$, we define
\begin{equation}\label{Lemma:Moment:E10}
	\|g\|_{\mathfrak S,T} \ := \ \sup_{0\le t< T} \|g(t)\|_{\mathfrak S}. 
\end{equation}
Let $T>0$ be a sufficient small constant and let us consider the set 
\begin{equation}\label{Lemma:Moment:E11}
	\mathscr X_T := \Big\{g\in C([0,T),\mathfrak S) ~~\Big|~~ \|g\|_{\mathfrak S,T}  \le 2 \|f_{in}\|_{\mathfrak S}\Big\}.
\end{equation}
We also define the operator \begin{equation}\label{Lemma:Moment:E12}\mathscr O[h] \ :=\ f_{in} \ +  \ \int_0^t\mathrm{d}s Q[h](s)|k|,
\end{equation}
which can be bounded, when $h\in\mathscr{X}_T$, as $$\|\mathscr O[h]\|_{\mathfrak S,T} \le \|f_{in}\|_{\mathfrak{S}} + TC_Q\|h\|_{\mathfrak S,T}^2 \le  \|f_{in}\|_{\mathfrak{S}} + T4C_Q\|f_{in}\|_{\mathfrak S}^2\le  2\|f_{in}\|_{\mathfrak{S}},$$
when $T$ is sufficiently small. Therefore the operator $\mathscr{O}$ maps $\mathscr{X}_T$ into $\mathscr{X}_T$.

It is also straightforward that, for any $h,g\in \mathfrak{S}$ and $\|h\|_\mathfrak{S},\|g\|_{\mathfrak{S}}\le 2 \|f_{in}\|_{\mathfrak S}$, we have $$\|\mathscr O[h]-\mathscr O[g]\|_{\mathfrak{S}}\le C\|h-g\|_{\mathfrak{S}},$$
where $C$ is a constant depending on $\|f_{in}\|_{\mathfrak S}$. Therefore, for $h$ and $g\in \mathscr X_T$ 
\begin{equation}\label{Lemma:Moment:E13}\left\| \int_0^t\mathrm{d}s Q[h](s)|k|-\int_0^{t}\mathrm{d}s Q[g](s)|k|\right\|_{\mathfrak S} \le CT\|h-g\|_{\mathfrak{S}}.
\end{equation}
As a consequence, the operator $\mathscr O[h]$ is Lipschitz from $\mathscr{X}_T$ to  $\mathscr{X}_T$. 

Next, we will show that if $\{h_n\}_{n=0}^\infty$ is a Cauchy sequence in $\mathscr{X}_T$, it will have a limit $h$ in  $\mathscr{X}_T$ as $n$ goes to infinity. This can be seen as follows, since $\{h_n\}_{n=0}^\infty$ is a Cauchy sequence in $\mathscr{X}_T$, then for all $t\in[0,T)$ and for a.e. $\widehat{k}$ in $\mathbb{S}$, 
$\{h_n\}_{n=0}^\infty$ is a Cauchy sequence in $\mathfrak S$. Therefore, for fixed $t\in[0,T)$ and $\widehat{k}\in \mathbb{S}$, there exists a subsequence $\{h_{n_i}\}_{i=0}^\infty$ such that $\{h_{n_i}\}_{i=0}^\infty$ converges weakly to a limit  $h\in\mathcal M_+([0,\infty))$ as $i$ goes to infinity in the weak topology of $\mathcal M_+([0,\infty)).$ By Lemma \ref{Lemma:Close}, $h\in\mathfrak{S}$. Moreover, 
$\lim_{i\to\infty}\int_{\{\Upsilon_\eta(\mu,\nu)\}}\mathrm{d}|k|h_{n_i}=\int_{\{\Upsilon_\eta(\mu,\nu)\}}\mathrm{d}|k|h$, for each $ \mu\ =\  1, 2, 3,\cdots$, $ \nu=1,\cdots, \Xi-1,$ $\eta=0,1,2,3,\cdots$. Therefore, the whole sequence $\{h_{n}\}_{n=0}^\infty$ converges weakly to    $h$ in  $\mathcal M_+([0,\infty))$ and $\lim_{n\to\infty}\int_{\{\Upsilon_\eta(\mu,\nu)\}}\mathrm{d}|k|h_{n}=\int_{\{\Upsilon_\eta(\mu,\nu)\}}\mathrm{d}|k|h$, for each $ \mu\ =\  1, 2, 3,\cdots$ ; $ \nu=1,\cdots, \Xi-1;$ \ \ $\eta=0,1,2,3,\cdots$. By Lemma \ref{Lemma:Close} again,  $h\in \mathscr{X}_T$ and hence $h$ is also a strong limit of $\{h_n\}_{n=0}^\infty$ in $ \mathscr{X}_T$. The claim is proved.

As a consequence, by the standard fixed point argument, the operator $ \mathscr O[h] $  has a fixed point, which is a mild local solution of our equation.

\end{proof}

\section{Moment estimates}

\begin{lemma}
	\label{Lemma:Moment} Suppose that $f\in C\left(  \left[  0,\infty\right),\mathfrak{S}\right)  $ solves \eqref{3wave} and \eqref{3wavenew} in the sense
	of Definition \ref{Def:Mild}.

	Then, the following inequality holds for $M=0,1,2,\cdots$ and a.e. $\widehat k\in\mathbb S$
\begin{equation}\begin{aligned}\label{Lemma:Moment:1}
			\sum_{\rho>M}  \int_{\mathbb{R}_+}{\rm d}|k| \sup_{\widehat{k}\in\mathbb{S}} f_\rho(t,|k|\widehat{k})\
		\ \le  \  & e^{\mathcal{C}_1t}\left(\sum_{\rho>M}\sup_{\widehat{k}\in\mathbb{S}}\int_{\mathbb{R}_+}{\rm d}|k|  f_\rho(0,|k|\widehat{k})\right)
\end{aligned}\end{equation} 
for some constant $\mathcal C_1>0$ independent of $t, 
 \rho,M.$
\end{lemma}

\begin{proof}
		We bound, for a positive test function $\varphi(k)\in    L^\infty_{|k|}L^1_{\widehat{k}}$
	\begin{equation}\begin{aligned}\label{Lemma:Moment:E1}
			&		\int_{\left[
				0,\infty\right)  \times \mathbb{S}}   {\rm d}k\,  {Q}[f](k)\varphi(k)
			\\
			\ \le \  &\sum_{\eta,\xi=-1}^\infty2\int_{\mathbb{R}^2}{\rm d}k_1\int_{|k_2|>|k_1|}{\rm d}|k_2|\,\frac{1}{|k_{1}|}  f_\eta(k_1)f_\xi(|k_2|\widehat{k_1})  \Big[\varphi(|k_1|\widehat{k_1}+|k_2|\widehat{k_1})+\varphi(-|k_1|\widehat{k_1}+|k_2|\widehat{k_1})\Big]\,\\
			&+\sum_{\eta,\xi=-1}^\infty \int_{\mathbb{R}^2}{\rm d}k_1\int_{|k_2|=|k_1|}{\rm d}|k_2|\,\frac{1}{|k_{1}|}  f_\eta(k_1)f_\xi(|k_2|\widehat{k_1})  \Big[\varphi(2|k_1|\widehat{k_1})+2\varphi(|k_1|\widehat{k_1})+2\varphi(0)\Big]\,.
	\end{aligned}\end{equation} 

 For $\rho\ge0$, we set $\varphi_{\mu,\nu}^\rho(k)=\delta_{|k|=	\Upsilon_\rho(\mu,\nu) }\tilde{\varphi}(\widehat{k_1}), $   for any $\tilde{\varphi}(\widehat{k_1})\in L^1(\mathbb{S})$, $\tilde{\varphi}(\widehat{k_1})\ge0$ and use those functions as test functions in \eqref{Lemma:Moment:E1}
 	\begin{equation}\begin{aligned}\label{Lemma:Moment:E2}
& \sum_{\rho>M}\sum_{ \mu\ =\  1}^{\infty} \sum_{  \nu=1}^{ \Xi-1}
 	\int_{\left[
 	0,\infty\right)  \times \mathbb{S}}   {\rm d}k\,  {Q}[f](k)\varphi^\rho_{\mu,\nu}(k)
 \\
 \ \le \  &\sum_{\rho>M}\sum_{ \mu\ =\  1}^{\infty} \sum_{  \nu=1}^{ \Xi-1}\sum_{\eta,\xi=-1}^\infty\int_{\mathbb{R}^2}{\rm d}k_1\int_{|k_2|>|k_1|}{\rm d}|k_2|\,\frac{1}{|k_{1}|}  f_\eta(k_1)f_\xi(|k_2|\widehat{k_1})  \Big[\varphi^\rho_{\mu,\nu}(|k_1|\widehat{k_1}+|k_2|\widehat{k_1})\\
 &+\varphi^\rho_{\mu,\nu}(-|k_1|\widehat{k_1}+|k_2|\widehat{k_1})\Big]\,\\
 &+ \sum_{\rho>M}\sum_{ \mu\ =\  1}^{\infty} \sum_{  \nu=1}^{ \Xi-1} \sum_{\eta,\xi=-1}^\infty \int_{\mathbb{R}^2}{\rm d}k_1\int_{|k_2|=|k_1|}{\rm d}|k_2|\,\frac{1}{|k_{1}|}  f_\eta(k_1)f_\xi(|k_2|\widehat{k_1}) \\
 &\times \Big[\varphi^\rho_{\mu,\nu}(2|k_1|\widehat{k_1})+2\varphi^\rho_{\mu,\nu}(0)\Big]\,.
\end{aligned}\end{equation} 
We   bound 
	\begin{equation}\begin{aligned}\label{Lemma:Moment:E3}
& \sum_{\rho>M}\sum_{ \mu\ =\  1}^{\infty} \sum_{  \nu=1}^{ \Xi-1}
\int_{\left[
	0,\infty\right)  \times \mathbb{S}}   {\rm d}k\,  {Q}[f](k)\varphi^\rho_{\mu,\nu}(k)
\\
		\ \lesssim \  &\sum_{\rho>M}\sum_{ \mu\ =\  1}^{\infty} \sum_{  \nu=1}^{ \Xi-1} \sum_{\eta,\xi=-1}^\infty\int_{\mathbb{R}^2}{\rm d}k_1\int_{|k_2|\ge |k_1|}{\rm d}|k_2|\,\frac{1}{|k_{1}|}  f_\eta(k_1)f_\xi(|k_2|\widehat{k_1})  \Big[\varphi^\rho_{\mu,\nu}(|k_1|\widehat{k_1}+|k_2|\widehat{k_1})\\
		&+\varphi^\rho_{\mu,\nu}(-|k_1|\widehat{k_1}+|k_2|\widehat{k_1})\Big]=: A +B.
\end{aligned}\end{equation} 
We now estimate $A$

	\begin{equation}\begin{aligned}\label{Lemma:Moment:E4}
	 A
		\ = \  &\sum_{\rho>M}\sum_{ \mu\ =\  1}^{\infty} \sum_{  \nu=1}^{ \Xi-1} \sum_{\eta,\xi=-1}^\infty\int_{\mathbb{R}^2}{\rm d}k_1\int_{|k_2|\ge |k_1|}{\rm d}|k_2|\,\frac{1}{|k_{1}|}  f_\eta(k_1)f_\xi(|k_2|\widehat{k_1})\delta_{|k_1|+|k_2|=	\Upsilon_\rho(\mu,\nu) }\tilde{\varphi}(\widehat{k_1}).
\end{aligned}\end{equation} 

Since $|k_1|+|k_2|=	\Upsilon_\rho(\mu,\nu)$, we suppose $|k_1|=\Upsilon_\eta(\mu'',\nu)$, $|k_2|=\Upsilon_\xi(\mu',\nu)$, $|k_1|+|k_2|=\Upsilon_\rho(\mu,\nu)$, 
then
	\begin{equation}\label{Lemma:Moment:E4}\Xi^{\rho-\eta}\Upsilon(\mu'',\nu)\ +\ \Xi^{\rho-\xi}\Upsilon(\mu',\nu) \ = \ \Upsilon(\mu,\nu).\end{equation} 
	Equation \eqref{Lemma:Moment:E4}  only has a non-empty set of solutions only when   $\eta=\xi\ge \rho$, $\eta=\rho\ge \xi$, or $\xi=\rho\ge \eta$. We then write
		\begin{equation}\begin{aligned}\label{Lemma:Moment:E5}
			A
			\ = \  & \sum_{\rho>M}\sum_{ \mu\ =\  1}^{\infty} \sum_{  \nu=1}^{ \Xi-1}  \sum_{\eta\ \ge \rho}\int_{\mathbb{R}^2}{\rm d}k_1\int_{|k_2|\ge |k_1|}{\rm d}|k_2|\,\frac{1}{|k_{1}|}  f_\eta(k_1)f_\eta(|k_2|\widehat{k_1}) \delta_{|k_1|+|k_2|=	\Upsilon_\rho(\mu,\nu) }\tilde{\varphi}(\widehat{k_1})\\
			  & +2\sum_{\rho>M}\sum_{ \mu\ =\  1}^{\infty} \sum_{  \nu=1}^{ \Xi-1} \sum_{-1\le\eta\le \rho}\int_{\mathbb{R}^2}{\rm d}k_1\int_{|k_2|\ge |k_1|}{\rm d}|k_2|\,\frac{1}{|k_{1}|}  f_\eta(k_1)f_\rho(|k_2|\widehat{k_1})  \\
			&\times  \delta_{|k_1|+|k_2|=	\Upsilon_\rho(\mu,\nu) }\tilde{\varphi}(\widehat{k_1}) \ =: \  A_1 \ + \ A_2.
	\end{aligned}\end{equation}

	We now estimate $A_1$ 		\begin{equation}\begin{aligned}\label{Lemma:Moment:E6}
			A_1
			\ \lesssim \  &\sum_{\rho>M}\sum_{ \mu\ =\  1}^{\infty} \sum_{  \nu=1}^{ \Xi-1}   \sum_{\eta\ge \rho}\int_{\mathbb{R}^2}{\rm d}k_1\int_{|k_2|\ge |k_1|}{\rm d}|k_2|\,\frac{1}{|k_{1}|}  f_\eta(k_1)f_\eta(|k_2|\widehat{k_1})\\
			&\times  \delta_{|k_1|+|k_2|=	\Upsilon_\rho(\mu,\nu) }\tilde{\varphi}(\widehat{k_1})\\
		\ \lesssim \  &		\sum_{\rho>M} \sum_{\eta\ge \rho}\left(\int_{\mathbb{S}}\mathrm d\widehat k\iint_{|k_2|\ge |k_1|}{\rm d}|k_2|\,{\rm d}|k_1|\, f_\eta(k_1) f_\eta(|k_2|\widehat{k_1})\tilde{\varphi}(\widehat{k_1})\right)
		\\	\ \lesssim \  &	 \left(\sum_{\eta>M}\sup_{\widehat{k}\in\mathbb{S}}\int_{\mathbb{R}_+}{\rm d}|k|  f_\eta(|k_1|\widehat{k})\right)\left(\sum_{\eta>M}\int_{\mathbb{S}}\mathrm d\widehat k\int_{\mathbb{R}_+}{\rm dk}  f_\eta(|k_1|\widehat{k})\tilde{\varphi}(\widehat{k})\right) \\	\ \lesssim \  &	 \left(\sum_{\rho>M}\int_{\mathbb{S}}\mathrm d\widehat k\int_{\mathbb{R}_+}{\rm d}|k|  f_\rho(|k|\widehat{k})\tilde{\varphi}(\widehat{k})\right),
	\end{aligned}\end{equation} 
	and $A_2$
		\begin{equation}\begin{aligned}\label{Lemma:Moment:E7}
			A_2
			\ \lesssim \  &  \sum_{\rho>M} \left(\sum_{-1\le\eta\le \rho}\sup_{\widehat{k}\in\mathbb{S}}\int_{\mathbb{R}_+}{\rm d}|k|  f_\eta(t,|k|\widehat{k})\right)\left(\int_{\mathbb{S}}\mathrm d\widehat k\int_{\mathbb{R}_+}{\rm d}|k|  f_\rho(|k_1|\widehat{k})\tilde{\varphi}(\widehat{k})\right)\\
				\ \lesssim \  &  \left(\sum_{\rho>M}\int_{\mathbb{S}}\mathrm d\widehat k\int_{\mathbb{R}_+}{\rm d}|k|  f_\rho(|k_1|\widehat{k})\tilde{\varphi}(\widehat{k})\right).
	\end{aligned}\end{equation} 
Similarly, we can also bound
	\begin{equation}\begin{aligned}\label{Lemma:Moment:E7a}
		B
		\ \lesssim \  &\left(\sum_{\rho>M}\int_{\mathbb{S}}\mathrm d\widehat k\int_{\mathbb{R}_+}{\rm d}|k|  f_\rho(|k_1|\widehat{k})\tilde{\varphi}(\widehat{k})\right).
\end{aligned}\end{equation} 
	Combining \eqref{Lemma:Moment:E2}, \eqref{Lemma:Moment:E6}, \eqref{Lemma:Moment:E7} we get
		\begin{equation}\begin{aligned}\label{Lemma:Moment:E9}
	\sum_{\rho>M}\int_{\mathbb{R}^2}{\rm d}k \partial_t f_\rho(t ,|k|\widehat{k})\tilde{\varphi} \ 
			\ \le  \  & \mathcal C_1\sum_{\rho>M}\int_{\mathbb{S}}\mathrm d\widehat k\int_{\mathbb{R}_+}{\rm d}|k|  f_\rho(|k_1|\widehat{k})\tilde{\varphi}(\widehat{k}),
	\end{aligned}\end{equation} 
for some constant $\mathcal C_1>0$ independent of $t, 
\rho,M,$ yielding
	\begin{equation}\begin{aligned}\label{Lemma:Moment:E10}
		\sum_{\rho>M}\int_{\mathbb{S}}\mathrm d\widehat k\int_{\mathbb{R}_+}{\rm d}|k|  f_\rho(t,|k_1|\widehat{k})\tilde{\varphi}(\widehat{k}) \ 
		\ \lesssim \  & \left(\sum_{\rho>M}\int_{\mathbb{S}}\mathrm d\widehat k\int_{\mathbb{R}_+}{\rm d}|k|  f_\rho(0,|k_1|\widehat{k})\tilde{\varphi}(\widehat{k})\right)e^{\mathcal C_1t}.
\end{aligned}\end{equation} 

Inequality \eqref{Lemma:Moment:1} then follows from \eqref{Lemma:Moment:E10}. 
\end{proof}

\begin{lemma}
\label{Lemma:Series} Recalling \eqref{Assum1}-\eqref{Assum2}, for  $M\in\mathbb{Z}, M>0$, $\frac{\epsilon\mathscr C_1 (M+1)^\gamma }{\mathscr C_2\mathscr C_3 }>1$ and \begin{equation}\begin{aligned}\label{Lemma:Series:1}
		t \ \le\  \mathcal T_{\epsilon,M}\ :=\  \frac{1}{\mathcal{C}_1}\ln\left(\frac{\epsilon\mathscr C_1 (M+1)^\gamma }{\mathscr C_2\mathscr C_3 }\right)\end{aligned}\end{equation} 
we have 	
\begin{equation}\begin{aligned}\label{Lemma:Series:2}
		\sum_{\rho>M}\sup_{\widehat{k}\in\mathbb{S}}\int_{\mathbb{R}_+}{\rm d}|k|  f_\rho(t,|k|\widehat{k})
		\ \le  \  &  \epsilon\sup_{\widehat{k}\in\mathbb{S}}\int_{\mathbb{R}_+}{\rm d}|k|  f_M(0,|k|\widehat{k}),
\end{aligned}\end{equation} 
for a.e. $\widehat k\in\mathbb S$.
\end{lemma}
\begin{proof}
	From \eqref{Lemma:Moment:1}, we get
	\begin{equation}\begin{aligned}\label{Lemma:Series:E1}
				\sum_{\rho>M}\sup_{\widehat{k}\in\mathbb{S}}\int_{\mathbb{R}_+}{\rm d}|k|  f_\rho(t,|k_1|\widehat{k})
			\ \le  \  &  e^{\mathcal{C}_1t}\left(\sum_{\rho>M}\sup_{\widehat{k}\in\mathbb{S}}\int_{\mathbb{R}_+}{\rm d}|k|  f_\rho(0,|k_1|\widehat{k})\right).
	\end{aligned}\end{equation} 

In order to have
	\begin{equation}\begin{aligned}\label{Lemma:Series:E3}
		\sum_{\rho>M}\sup_{\widehat{k}\in\mathbb{S}}\int_{\mathbb{R}_+}{\rm d}|k|  f_\rho(t,|k_1|\widehat{k})
		\ \le  \  &  \epsilon\sup_{\widehat{k}\in\mathbb{S}}\int_{\mathbb{R}_+}{\rm d}|k|  f_M(0,|k_1|\widehat{k}),
\end{aligned}\end{equation} 
we would need
	\begin{equation}\begin{aligned}\label{Lemma:Series:E4}
e^{\mathcal{C}_1t}\left(\sum_{\rho>M}\sup_{\widehat{k}\in\mathbb{S}}\int_{\mathbb{R}_+}{\rm d}|k|  f_\rho(0,|k_1|\widehat{k})\right)	  \ \le \ &  \epsilon\sup_{\widehat{k}\in\mathbb{S}}\int_{\mathbb{R}_+}{\rm d}|k|  f_M(0,|k_1|\widehat{k}),
\end{aligned}\end{equation} 
which is equivalent to 
	\begin{equation}\begin{aligned}\label{Lemma:Series:E5}
		  t \ \le\ & \frac{1}{\mathcal{C}_1}\ln\left(\frac{\epsilon\sup_{\widehat{k}\in\mathbb{S}}\int_{\mathbb{R}_+}{\rm d}|k|  f_M(0,|k_1|\widehat{k})}{\sum_{\rho>M}\sup_{\widehat{k}\in\mathbb{S}}\int_{\mathbb{R}_+}{\rm d}|k|  f_\rho(0,|k_1|\widehat{k})}\right)\\
		\ \ \le \ & \frac{1}{\mathcal{C}_1}\ln\left(\frac{\epsilon\mathscr C_1 \frac{\mathscr C_3^M}{(M!)^\gamma} }{\mathscr C_2\left( \frac{\mathscr C_3^{M+1}}{((M+1)!)^\gamma}\ +\ \frac{\mathscr C_3^{M+2}}{((M+2)!)^\gamma}\ +\  \cdots \right)}\right)\\
		\ \ \le \ & \frac{1}{\mathcal{C}_1}\ln\left(\frac{\epsilon\mathscr C_1 \frac{\mathscr C_3^M}{(M!)^\gamma} }{\mathscr C_2 \frac{\mathscr C_3^{M+1}}{((M+1)!)^\gamma}}\right)\ \le \  \frac{1}{\mathcal{C}_1}\ln\left(\frac{\epsilon\mathscr C_1 (M+1)^\gamma }{\mathscr C_2\mathscr C_3 }\right) =\mathcal T_{\epsilon,M},
\end{aligned}\end{equation} 
when $\frac{\epsilon\mathscr C_1 (M+1)^\gamma }{\mathscr C_2\mathscr C_3 }>1.$ Note that we have used \eqref{Assum1}. 
\end{proof}

\begin{proposition}\label{Propo:Selfsimilar}
	We  have the bound for $f$ being a mild solution of  \eqref{3wave} and \eqref{3wavenew} 
\begin{equation}
	\label{Lemma:Delta:2}
	\|f(t)\|_{\mathfrak S} \ \le \ {\mathcal C_3e^{\mathcal{C}_1t}}, \ \ \ \forall t\ge 0,
\end{equation}
for some universal constant $\mathcal C_3>0$ independent of $t$. 	There exists a global mild solution of the equation \eqref{3wave} and \eqref{3wavenew} in $C([0,\infty),\mathfrak S)$ in the sense of Definition \ref{Def:Mild}. 	
	
		\end{proposition}
\begin{proof}
	 Applying \eqref{Lemma:Moment:1}, we bound, for all $M\ge 0$ 
	\begin{equation}\begin{aligned}\label{Lemma:Delta:E11}
			&	 \int_{\mathbb{R}_+}{\rm d}|k| \sup_{\widehat{k}\in\mathbb{S}} f_M(t,|k|\widehat{k})\
			\ \le  \  e^{\mathcal{C}_1t}\left(\sum_{\rho\ge M}\sup_{\widehat{k}\in\mathbb{S}}\int_{\mathbb{R}_+}{\rm d}|k|  f_\rho(0,|k|\widehat{k})\right)\\
			\ \le  \  & e^{\mathcal{C}_1t}\left(\sum_{\rho\ge M}\mathscr C_1\frac{\mathscr C_3^\rho}{(\rho!)^\gamma}\right) 	\ \le  \  e^{\mathcal{C}_1t}\left(\mathscr C_1\frac{\mathscr C_3^M}{(M!)^\gamma }S_\gamma(\mathscr C_3)\right),
	\end{aligned}\end{equation}
	where   
	\begin{equation}\begin{aligned}\label{Lemma:Delta:E11a}
			S_\gamma(x)\ :=\ 	&   \sum_{n=0}^\infty\frac{x^n}{(n!)^\gamma},\ \ \ \ \ \ \ \  \forall x> 0.
	\end{aligned}\end{equation}

	As a result, we  obtain
	\begin{equation}\begin{aligned}\label{Lemma:Delta:E12}
			&	\mathfrak{C}^M \int_{\mathbb{R}_+}{\rm d}|k| \sup_{\widehat{k}\in\mathbb{S}} f_M(t,|k|\widehat{k})\
			\ \le  \  \mathfrak{C}^M e^{\mathcal{C}_1t}\left(\mathscr C_1\frac{\mathscr C_3^M	}{(M!)^\gamma}S_\gamma(\mathscr C_3)\right)	\ \le  \  \mathcal C_3e^{\mathcal{C}_1t},
	\end{aligned}\end{equation}
	for some universal constant $\mathcal C_3>0$ independent of $M$. From \eqref{Lemma:Delta:E12}, we deduce
	\begin{equation}
		\label{Lemma:Delta:E13}
		\|f(t)\|_{\mathfrak S} \ \le \ \mathcal C_3e^{\mathcal{C}_1t},
	\end{equation}
	for all $t\ge0$. Therefore \eqref{Lemma:Delta:2} is proved.

	The global existence result follows from \eqref{Lemma:Delta:2}  and \eqref{Assum2} as well as iterating the fixed point argument of Proposition \ref{Lemma:WeakFormulationb}. 
	
	\end{proof}
\section{Long time limit}
\begin{lemma}\label{Lemma:TestLayer}  Suppose that $f\in C\left(  \left[  0,\infty\right),\mathfrak{S}\right)  $ solves \eqref{3wave} and \eqref{3wavenew} in the sense
	of Definition \ref{Def:Mild}.  For  $\widehat{k}\in\mathbb{S}$ suppose  Alternative (I) of Theorem \ref{maintheorem} does not hold true.	
	We bound $\forall \epsilon\in[0,1)$,
	\begin{align}\label{Lemma:TestLayer:1}
		\begin{split}
			&	\left(	\sum_{\eta\ge\rho} \int_{\Theta_\eta}\mathrm{d}|k|{f_\eta(t,k)}\Big((1+\epsilon)\Xi^{-\rho}-|k|\Big)\right)\
			-	\left(	\sum_{\eta\ge\rho} \int_{\Theta_\eta}\mathrm{d}|k|{f_\eta(0,k)}\Big((1+\epsilon)\Xi^{-\rho}-|k|\Big)\right)\\
			\ \ge \  &\mathcal{C}_2(1-\epsilon)\Xi^{-\rho}\int_0^t\mathrm d s\left(\int_{|k|=\Xi^{-\rho}}{\rm d}|k|f_\rho(s,\Xi^{-\rho}\widehat{k})\right)^2
		\end{split}
	\end{align}
	and for some constant $\mathcal C_2>0$ independent of $t, 
	\rho, g,\epsilon$.
\end{lemma}
\begin{proof}
Since Alternative (I) of Theorem \ref{maintheorem} does not hold true,   for all $t\ge0$
	\begin{equation}\label{Lemma:TestLayer:E0}\int_{\{0\}}\mathrm d|k|f(t,|k|\widehat k) \ = \ 0.\end{equation}
	From \eqref{Lemma:WeakFormulation:E7}, with $\varphi(k)=\varphi(|k|)$, we find 	\begin{align}\label{Lemma:TestLayer:E1}
		\begin{split}
			\partial_t	\int_{\mathbb{R}^2} {\rm d}k\,  \frac{f}{|k|}\varphi(k)
			\ = \  &2\int_{\mathbb{R}^2}{\rm d}k_1\int_{|k_2|>|k_1|}{\rm d}|k_2|\,\frac{1}{|k_{1}|}  f(k_1)f(|k_2|\widehat{k_1})  \Big[ \varphi(|k_1|+|k_2|)\\
			&-2\varphi(|k_2|)+\varphi(-|k_1|+|k_2|)\Big]\,\\
			&+ \int_{\mathbb{R}^2}{\rm d}k_1\int_{|k_2|=|k_1|}{\rm d}|k_2|\,\frac{1}{|k_{1}|}  f(k_1)f(|k_2|\widehat{k_1})  \Big[\varphi(2|k_1|)-2\varphi(|k_1|)+2\varphi(0)\Big]\, .
		\end{split}
	\end{align}
	
	By using the test function $\varphi_c(|k|)=(c-|k|)_+$, $c>0$, which satisfies $\varphi_c(|k|)=c-|k|$ when $|k|\le c$ and $\varphi_c(|k|)=0$ when $|k|> c$. Let us study the quantities
	\begin{equation}\begin{aligned}\label{Lemma:Coercive:E2}\mathcal L_c^1 \ := \ &\varphi_c(|k_1|+|k_2|)\ - \ 2\varphi_c(|k_2|)+\varphi_c(-|k_1|+|k_2|),\end{aligned} \end{equation}
	and
	\begin{equation}\label{Lemma:Coercive:E3}\mathcal L_c^2 \ := \ \varphi_c(2|k_1|)-2\varphi_c(|k_1|)+2c. \end{equation}
	We first prove that $\mathcal L_c^1\ge0$. We consider several cases.
	\begin{itemize}
		\item  If $|k_2|\ge|k_2|-|k_1|
		\ge c$ or $|k_2|+|k_1|\le c$, then
		 \begin{equation}\begin{aligned}\mathcal L_c^1 \ =\ &  0.\end{aligned} \end{equation}
		\item  If $|k_2|+|k_1|>c$ and $|k_2|\le c$
		\begin{equation}\begin{aligned}\mathcal L_c^1 \ = \ &\ - \ 2(c-|k_2|)+(c+|k_1|-|k_2|)>0\end{aligned}. \end{equation}
		\item  If $|k_2|+|k_1|>c$ and $|k_2|>c$
		\begin{equation}\begin{aligned}\mathcal L_c^1 \ = \ &\varphi_c(-|k_1|+|k_2|)\ \ge \ 0.\end{aligned} \end{equation}
		
	\end{itemize}
	
	Next, it is straightforward that $\mathcal L_c^2=(c-2|k_1|)_+-2(c-|k_1|)_++2c\ge0$. Choosing $\varphi=\varphi_{(1+\epsilon)\Xi^{-\rho}}(|k|)\hat\varphi(\widehat{k})$, for $\rho\ge0$, $\rho\in\mathbb{Z}$, $\hat\varphi(\widehat{k})\in L^1(\mathbb{S})$, $\hat\varphi(\widehat{k})\ge 0$, we bound
	\begin{align}\label{Lemma:TestLayer:E2}
		\begin{split}
			&	\partial_t	\int_{\mathbb{R}^2} {\rm d}k\,  \frac{f}{|k|}\Big((1+\epsilon)\Xi^{-\rho}-|k|\Big)_+\hat\varphi(\widehat{k})\\
			\ \ge \  &2\int_{\mathbb{R}^2}{\rm d}k_1\int_{|k_2|\ge |k_1|}{\rm d}|k_2|\,\frac{1}{|k_{1}|}  f(k_1)f(|k_2|\widehat{k_1})  \Big[ \Big((1+\epsilon)\Xi^{-\rho}-|k_1|-|k_2|\Big)_+\\
			&-2\Big((1+\epsilon)\Xi^{-\rho}-|k_2|\Big)_++\Big((1+\epsilon)\Xi^{-\rho}+|k_1|-|k_2|\Big)_+\Big]\hat\varphi(\widehat{k}_1)\\
			\ \gtrsim \  &(1-\epsilon)\Xi^{-\rho}\int_{\mathbb{S}}\mathrm{d}\widehat{k}\left(\int_{|k|=\Xi^{-\rho}}{\rm d}|k|f_\rho(\Xi^{-\rho}\widehat{k})\right)^2\hat\varphi(\widehat{k}),
		\end{split}
	\end{align}
	\begin{align}\label{Lemma:TestLayer:E3}
		\begin{split}
			&		\sum_{\eta\ge\rho} \int_{\Theta_\eta}\mathrm{d}|k| {f_\eta}(t)\Big((1+\epsilon)\Xi^{-\rho}-|k|\Big)_+{}-	\sum_{\eta\ge\rho} \int_{\Theta_\eta}\mathrm{d}|k| {f_\eta}(0)\Big((1+\epsilon)\Xi^{-\rho}-|k|\Big)_+{}\\
			\ \gtrsim \  &(1-\epsilon)\Xi^{-\rho}\int_0^t\mathrm ds\left(\int_{|k|=\Xi^{-\rho}}{\rm d}|k|f_\rho(s)(\Xi^{-\rho}\widehat{k})\right)^2.
		\end{split}
	\end{align}

\end{proof}

\begin{lemma}
	\label{Lemma:Coercive}Suppose that $f\in C\left(  \left[  0,\infty\right),\mathfrak{S}\right)  $ solves \eqref{3wave} and \eqref{3wavenew} in the sense
	of Definition \ref{Def:Mild}. For  $\widehat{k}\in\mathbb{S}$ suppose   Alternative (I) of Theorem \ref{maintheorem} does not hold true.		Then, the following inequalities hold for  $t\in\left[  0,\infty\right]  $ 	
	\begin{align}\label{Lemma:Coercive:0}
		\begin{split}
			\int_{(0,\infty)} {\rm d}|k|\,  {f(t,k)}\Big(c-|k|\Big)_+
			\ \mbox{ is nondecreasing in } t, \ \ \ \forall c>0,
		\end{split}
	\end{align}
	and 
	\begin{align}\label{Lemma:Coercive:1}
		\begin{split}
			\int_{(0,c]} {\rm d}|k|\,  {f(t,k)}  
			\ \gtrsim  \ 	&     	\int_{0}^t {\rm d}s\left( \int_{c\ge |k|\ge \frac{3c}{4}}{\rm d}|k|  f(s,k)\right)^2,
		\end{split}
	\end{align}
	where the constant on the right hand side of \eqref{Lemma:Coercive:1} is independent of $t,c$.
	
\end{lemma}

\begin{proof}
		
	By choosing $\varphi(k)=\varphi_c(k)\tilde\varphi(
	\widehat k)=\Big(c-|k|\Big)_+\tilde\varphi(
	\widehat k)\ge 0$, $\tilde\varphi(
	\widehat k)\in L^1(\mathbb S)$, we bound 
	
	\begin{align}\label{Lemma:Coercive:E4}
		\begin{split}
			&	\partial_t	\int_{\mathbb{R}^2} {\rm d}k\,  \frac{f}{|k|}\varphi(k)\\
			\ \ge  \ & 2\int_{\mathbb{R}^2}{\rm d}k_1\int_{|k_2|>|k_1|}{\rm d}|k_2|\,\frac{1}{|k_{1}|}  f(k_1)f(|k_2|\widehat{k_1}) \tilde\varphi(
			\widehat{k_1}) \Big[\varphi_c(|k_1|+|k_2|)\\
			&-2\varphi_c(|k_2|)+\varphi_c(-|k_1|+|k_2|)\Big]\\
			&+ \int_{\mathbb{R}^2}{\rm d}k_1\int_{|k_2|=|k_1|}{\rm d}|k_2|\,\frac{1}{|k_{1}|}  f(k_1)f(|k_2|\widehat{k_1}) \tilde\varphi(
			\widehat{k_1}) \Big[\varphi_c(2|k_1|)-2\varphi_c(|k_1|)+2\varphi_c(0)\Big]\\
			\ \ge  \ 	& 2\int_{\mathbb{R}^2}{\rm d}k_1\int_{|k_2|\ge|k_1|}{\rm d}|k_2|\,\frac{1}{|k_{1}|}  f(k_1)f(|k_2|\widehat{k_1}) \tilde\varphi(
			\widehat{k_1}) \Big[\varphi_c(|k_1|+|k_2|)\\
			&-2\varphi_c(|k_2|)+\varphi_c(-|k_1|+|k_2|)\Big]\\
			\ \gtrsim  \ 	& \int_{\mathbb{S}}\mathrm{d}\widehat{k_1} \iint_{c\ge |k_2|,|k_1|\ge \frac{3c}{4}}{\rm d}|k_1|{\rm d}|k_2|\,   f(k_1)f(|k_2|\widehat{k_1}) \tilde\varphi(
			\widehat{k_1})\Big[- \ 2(c-|k_2|)+(c+|k_1|-|k_2|)\Big] ,
		\end{split}
	\end{align}
	yielding, 
	\begin{align}\label{Lemma:Coercive:E4a}
		\begin{split}
			\partial_t	\int_{[0,\infty)} {\rm d}|k|\,  {f}\varphi_c(|k|) 
			\ \gtrsim  \ 	& \frac{c}{2} \iint_{c\ge |k_2|,|k_1|\ge \frac{3c}{4}}{\rm d}|k_1|{\rm d}|k_2|\,   f(k_1)f(|k_2|\widehat{k_1}),
		\end{split}
	\end{align}
	which leads to \eqref{Lemma:Coercive:0} by \eqref{Lemma:TestLayer:E0}. We now bound
	
	\begin{align}\label{Lemma:Coercive:E5}
		\begin{split}
			\partial_t	\int_{[0,\infty)} {\rm d}|k|\,  {f}\varphi_c(|k|)\ 
			\ \gtrsim  \ 		&   \frac{c}{2} \left( \int_{c\ge |k|\ge \frac{3c}{4}}{\rm d}|k|   f(k)\right)^2.
		\end{split}
	\end{align}Since $\varphi_c(|k|)\le c$, we bound using \eqref{Lemma:TestLayer:E0}
	\begin{align}\label{Lemma:Coercive:E6}
		\begin{split}
			c\int_{{(0,c]}} {\rm d}|k|\,   {f(t,k)} \ 
			\ \gtrsim  \ 	&    \frac{c}{2}	\int_{0}^t {\rm d}s\left( \int_{c\ge |k|\ge \frac{3c}{4}}{\rm d}|k|   f(s,k) \right)^2+ \int_{(0,\infty)} {\rm d}|k|\,  {{f(0,k)}}\varphi_c(|k|),
		\end{split}
	\end{align}
	yielding the desired estimate.  
\end{proof}

\begin{proposition}\label{Lemma:Cascade}
Suppose that $f\in C\left(  \left[  0,\infty\right),\mathfrak{S}\right)  $ solves \eqref{3wave} and \eqref{3wavenew} in the sense
of Definition \ref{Def:Mild}. For  $\widehat{k}\in\mathbb{S}$ suppose  Alternative (I) of Theorem \ref{maintheorem} does not hold true. Then
 for any $c>0$ \begin{equation}\label{Lemma:Cascade:1}
		\lim_{t\rightarrow\infty}\int_{(0,\infty)} {\rm d}|k|\,  {f(t)}\Big(1-\frac{|k|}{c}\Big)_+=\int_{(0,\infty)} {\rm d}|k|{f}(0),
	\end{equation} 
	and \begin{equation}\label{Lemma:Cascade:2}
		\lim_{t\rightarrow\infty}\int_{(0,c)} {\rm d}|k|\,  {f(t)}=\int_{(0,\infty)} {\rm d}|k|{f}(0).
	\end{equation}

	Moreover, there exist a time sequence $\{\tau_n\}_{n=1}^\infty$ and a constant $N_0>1$ such that $\tau_1<\tau_2<\cdots<\tau_n<\cdots$ and $\lim_{n\to\infty}\tau_n=\infty$, and for all $n>N_0$ and 	
	\begin{align}\label{Lemma:Delta:1}	\begin{split}
			&\int_{|k|=\Xi^{-n}}\mathrm{d}|k|{f(t,k)}  \
			\ >\ 	 \mathfrak{C}_1  (t+1)\int_{\{|k|\le\Xi^{-n}\}}\mathrm{d}|k|{f(0,k)}, 	\end{split}
	\end{align}
	for all $t\in[\tau_{n-1},\tau_n)$, where $\mathfrak C_1>0$ is a   constant independent of $t ,\widehat k$ and $n,N_0$.
	)

\end{proposition}
\begin{proof} 
	By \eqref{Lemma:CollisionDiscrete3}, we find $$\int_{(0,\infty)} {\rm d}|k|{f}(t,|k|\widehat k)\ \le \ \int_{(0,\infty)} {\rm d}|k|{f}(0,|k|\widehat k).$$
	
	Since the  mapping $t\mapsto 	\int_{\mathbb{R}_+} {\rm d}|k|\,  {f}(t) \Big(c-|k|\Big)_+$  	is  nondecreasing  by \eqref{Lemma:Coercive:0}
	and
	\begin{align}\label{Lemma:Cascade:E1}
		\begin{split}
			0\ \le\	\int_{(0,\infty)} {\rm d}|k|\,  {f}(t) \Big(c-|k|\Big)_+
			\ \le  \ c\int_{(0,\infty)} {\rm d}|k|{f}(t)\ \le \ c\int_{(0,\infty)} {\rm d}|k|{f}(0),
		\end{split}
	\end{align} 
	there exists a limit $$\mathfrak{E}(\widehat{k})=\frac1c\lim_{t\rightarrow\infty}\int_{(0,\infty)} {\rm d}|k|\,  {f}(t)\Big(c-|k|\Big)_+\le \int_{(0,\infty)} {\rm d}|k|{f}(0).$$ We will prove  $\mathfrak{E}(\widehat{k}) =\int_{(0,\infty)} {\rm d}|k|{f}(0),$ using a proof by contradiction.
	Supposing  $\int_{(0,\infty)} {\rm d}|k|{f}(0)-\mathfrak{E}(\widehat{k})= \varepsilon(\widehat{k})>0$, for a.e. $\widehat k\in\mathbb S$, we will prove the existence of   $0<r_2<r_1<\infty$ such that
	\begin{equation}\label{Lemma:Cascade:E2}
		\int_{[r_1,r_2]}{\rm d}|k|{}f(t)\ge C({\varepsilon(\widehat{k})})\text{ for all }t\in[0,\infty),
	\end{equation}
for some constant $C({\varepsilon(\widehat{k})})>0$.

For any $r_2\in(0,c)$, we bound
	\begin{equation}\label{Lemma:Cascade:E4}\begin{aligned}
			\int_{(0,r_2]}{\rm d}|k|{f}(t)\ & \leq\ \int_{(0,r_2]}\frac{\Big(c-|k|\Big)_+}{\Big(c-r_2\Big)_+}{\rm d}|k|{f}(t) \ \leq\ \int_{(0,\infty)}\frac{\Big(c-|k|\Big)_+}{\Big(c-r_2\Big)_+}{\rm d}|k|{f}(t)\\ 
			&\leq\ \frac{c\mathfrak{E}(\widehat{k})}{c-r_2}.\end{aligned}
	\end{equation}

	Next, we bound, using \eqref{Lemma:Coercive:0}, for $r_1>\max\{r_2,1\}$
	\begin{equation}\label{Lemma:Cascade:E6}
		\begin{split}
		r_1	\int_{(0,r_1]}{\rm d}|k|{f}(t)&\geq\int_{(0,\infty)}{\rm d}|k|\left(r_1-{ |k|}\right)_+{f(t)}\geq\int_{(0,\infty)}\left({r_1}-{ |k|}\right)_+{\rm d}|k|{f}(0)\\&\geq\left({r_1 }-{1}\right)\int_{(0,1]}{\rm d}|k|{f}(0) \ = \ \left({r_1 }-{1}\right)\int_{(0,\infty)}{\rm d}|k|{f}(0)
		\end{split}
	\end{equation}
with the notice that ${f}(0)$ is supported in $[0,1]$) $\text{ for all }t\in[0,\infty),$	which implies
	\begin{equation}\label{Lemma:Cascade:E7}
	\int_{(0,r_1]}{\rm d}|k|{f}(t) \ge \frac{r_1-1}{r_1}\int_{(0,\infty)}{\rm d}|k|{f}(0)\text{ for all }t\in(0,\infty).
	\end{equation}
	Choosing $r_1,r_2$ such that 
	\begin{equation}\label{Lemma:Cascade:E8}
		 \frac{c\mathfrak{E}(\widehat{k})}{c-r_2} <\frac{r_1-1}{r_1}\int_{(0,\infty)}{\rm d}|k|{f}(0),
	\end{equation}
	and combining \eqref{Lemma:Cascade:E4} and \eqref{Lemma:Cascade:E7}, we obtain  \eqref{Lemma:Cascade:E2}.
	Let $N$ be an integer and $r_0<r_1$ such that $[r_1,r_2]\subset \left[r_0,\frac{4^{N}}{3^{N}}r_0\right]$, we bound, using    \eqref{Lemma:Coercive:1}  
	
	\begin{equation}\label{Lemma:Cascade:E8a}
		\begin{aligned}
			\int_{\Big(0,\frac{4^{N}}{3^{N}}r_0\Big]}{\rm d}|k|{f}(t) \ \gtrsim & \  \int_0^t\mathrm d s\left[\int_{\left[\frac{4^{N-1}}{3^{N-1}}r_0,\frac{4^{N}}{3^{N}}r_0\right]}{\rm d}|k|{f}(s)\right]^2,\\
					\int_{\Big(0,\frac{4^{N-1}}{3^{N-1}}r_0\Big]}{\rm d}|k|{f}(t) \ \gtrsim & \  \int_0^t\mathrm d s\left[\int_{\left[\frac{4^{N-2}}{3^{N-2}}r_0,\frac{4^{N-1}}{3^{N-1}}r_0\right]}{\rm d}|k|{f}(s)\right]^2,\\
					& \cdots\\
					\int_{\Big(0,\frac{4}{3}r_0\Big]}{\rm d}|k|{f}(t) \ \gtrsim & \  \int_0^t\mathrm d s\left[\int_{\left[r_0,\frac{4}{3}r_0\right]}{\rm d}|k|{f}(s)\right]^2,
	\end{aligned}\end{equation}
	which, by summing all the inequalities and by the Cauchy-Schwarz inequality, yields 
	\begin{equation}\label{Lemma:Cascade:E9}
		\begin{aligned}
			\int_{(0,\infty)}{\rm d}|k|{f}(t) \ \gtrsim & \ r(N) \int_0^t\mathrm d s\left[\int_{\left[r_0,\frac{4^{N}}{3^{N}}r_0\right]}{\rm d}|k|{f}(s)\right]^2\\
			\ \gtrsim & \ r(N) \int_0^t\mathrm d s\left[\int_{\left[r_1,r_2\right]}{\rm d}|k|{f}(s)\right]^2\
			\gtrsim \ {r(N)C({\varepsilon(\widehat{k})})^2 t},
	\end{aligned}\end{equation}
where $r(N)$ is a constant depending on $N$ and coming from the Cauchy-Schwarz inequality. Inequality \eqref{Lemma:Cascade:E9} leads to
	\begin{equation*}
		\begin{aligned}
			\int_{(0,\infty)}{\rm d}|k|{f}(0) \ \ge & \int_{(0,\infty)}{\rm d}|k|{f}(t) 
						\ \gtrsim  \ {C({\varepsilon(\widehat{k})})^2r(N)t} \ \to \infty 
	\end{aligned}\end{equation*}
	as $t$ to $\infty$ which is a contradiction.  As a conclusion, $\mathfrak{E}(\widehat{k}) =\int_{\mathbb{R}_+} {\rm d}|k|{f}(0)$ and that implies \eqref{Lemma:Cascade:1} of the Proposition. 
	
	The second limit \eqref{Lemma:Cascade:2} is a consequence of the first limit \eqref{Lemma:Cascade:1}.

	From \eqref{Lemma:TestLayer:1}, we obtain, for $0<\epsilon<1$ 
\begin{align}\label{Lemma:Delta:E3}
	\begin{split}
		&			\sum_{\eta>\rho} \int_{\Theta_\eta}\mathrm{d}|k| {f_\eta}(t,k)\Big((1+\epsilon)\Xi^{-\rho}-|k|\Big)_+{}	\ + \ \epsilon\Xi^{-\rho}\int_{|k|=\Xi^{-\rho}}\mathrm{d}|k|{f_\rho(t,k)}  \\
		\ \ge \  &\mathcal{C}_2(1-\epsilon)\Xi^{-\rho}\int_{0}^t\mathrm{d}s\left(\int_{|k|=\Xi^{-\rho}}{\rm d}|k|f_\rho(s,\Xi^{-\rho}\widehat{k})\right)^2\\
		&	+ \ 	\sum_{\eta>\rho} \int_{\Theta_\eta}\mathrm{d}|k|{f_\eta(0,k)}\Big((1+\epsilon)\Xi^{-\rho}-|k|\Big)_+\ + \ \epsilon{\Xi^{-\rho}}\int_{|k|=\Xi^{-\rho}}\mathrm{d}|k|{f_\rho(0,k)}.
	\end{split}
\end{align}

Next, we bound, using \eqref{Lemma:Series:2} 
	\begin{align}\label{Lemma:Delta:E2}
	\begin{split}
		&	\sum_{\eta>\rho} \int_{\Theta_\eta}\mathrm{d}|k|{f_\eta(t,k)}\Big((1+\epsilon)\Xi^{-\rho}-|k|\Big)_+{} \ \le \   \sup_{\widehat{k}\in\mathbb{S}}\Xi^{-\rho}\frac{\mathscr{C}_2\epsilon}{\mathscr{C}_12}  \int_{|k|=\Xi^{-\rho}}{\rm d}|k| f_\rho(0,\Xi^{-\rho}\widehat{k})\\
		&  \ \le \  \Xi^{-\rho} \frac\epsilon2\int_{|k|=\Xi^{-\rho}}{\rm d}|k| f_\rho(0,\Xi^{-\rho}\widehat{k}),
	\end{split}
\end{align}
for $0\le \tau\le \mathcal T_{\frac{\mathscr{C}_2\epsilon}{\mathscr{C}_14}  ,
	\rho}$. 
By \eqref{Lemma:Delta:E2}, we estimate
\begin{align}\label{Lemma:Delta:E4}
	\begin{split}
		&	\epsilon\Xi^{-\rho}\int_{|k|=\Xi^{-\rho}}\mathrm{d}|k|{f_\rho(t,k)} \
		\ \ge \  \mathcal{C}_2(1-\epsilon)\Xi^{-\rho}\int_{0}^t\mathrm{d}s\left(\int_{|k|=\Xi^{-\rho}}{\rm d}|k|f_\rho(s,\Xi^{-\rho}\widehat{k})\right)^2\\
		&	+ \ 	\sum_{\eta>\rho} \int_{\Theta_\eta}\mathrm{d}|k|{f_\eta(0,k)}\Big((1+\epsilon)\Xi^{-\rho}-|k|\Big)_+{} \ + \ \frac\epsilon2{\Xi^{-\rho}}\int_{|k|=\Xi^{-\rho}}\mathrm{d}|k|{f_\rho(0,k)},
	\end{split}
\end{align}
yielding, for $\epsilon=\frac{1}{2}$, 
\begin{align}\label{Lemma:Delta:E5}	\begin{split}
	&			 \int_{|k|=\Xi^{-\rho}}\mathrm{d}|k|{f_\rho(t,k)}  \
		\ \ge \  \mathcal{C}_2\int_{0}^t\mathrm{d}s\left(\int_{|k|=\Xi^{-\rho}}{\rm d}|k| f_\rho(s,\Xi^{-\rho}\widehat{k})\right)^2\ 	+ \ \frac12\int_{|k|=\Xi^{-\rho}}\mathrm{d}|k|{f_\rho(0,k)}.
	\end{split}
\end{align}
Setting $$X(t,\widehat{k})=\int_{0}^t\mathrm{d}s\left(\int_{|k|=\Xi^{-\rho}}{\rm d}|k| f_\rho(s,\Xi^{-\rho}\widehat{k})\right)^2$$ and $X_o(\widehat{k})=
 	  \ \int_{|k|=\Xi^{-\rho}}\mathrm{d}|k|{f_\rho(0,k)}$, we bound 
 	  \begin{align}\label{Lemma:Delta:E6}	\begin{split}
 	  	\sqrt{\dot{X}(t)}\ \ge\ 	&\mathcal{C}_2	X(t)\ + \ \frac12X_o, 
 	  	\end{split}
 	  \end{align}
   yielding
     \begin{align}\label{Lemma:Delta:E7}	\begin{split}
   		{\dot{X}(t)} \ \ge\ 	&	\Big(\mathcal{C}_2X(t)\ + \ \frac12X_o\Big)^2. 
   	\end{split}
   \end{align}
Solving  \eqref{Lemma:Delta:E7}, we find
\begin{align}\label{Lemma:Delta:E8}	\begin{split}
		{{X}(t)} \ \ge\ 	&  \frac{t\frac12X_0}{1-t\frac12\mathcal{C}_2X_0},
	\end{split}
\end{align}
for $0\le t\le \mathcal T_{\frac{\mathscr{C}_2\epsilon}{\mathscr{C}_14}  ,\rho}$, which implies 
\begin{align}\label{Lemma:Delta:E9}	\begin{split}
\int_{|k|=\Xi^{-\rho}}\mathrm{d}|k|{f_\rho(t,k)}  \
 \ \ge\ 	&\frac{t\frac12X_0\mathcal{C}_2}{1-t\frac12\mathcal{C}_2X_0}\ + \ \frac12 X_o. 
	\end{split}
\end{align}
When $t\in\big[ \mathcal T_{\frac{\mathscr{C}_2\epsilon}{\mathscr{C}_14}  ,
	\rho-1},\mathcal T_{\frac{\mathscr{C}_2\epsilon}{\mathscr{C}_14}  ,
	\rho}\big)$, we can set  $\tau_{\rho}=\mathcal T_{\frac{\mathscr{C}_2\epsilon}{\mathscr{C}_14}  ,
	\rho}$ and bound for $\rho$ sufficiently large
 \begin{align}\label{Lemma:Delta:E10}	\begin{split}
 	&	\int_{|k|=\Xi^{-\rho}}\mathrm{d}|k|{f_\rho(t,k)}  \
 		 \ >\ 	 \mathfrak{C}_1 (t+1)\int_{\{|k|\le\Xi^{-\rho}\}}\mathrm{d}|k|{f_\rho(0,k)} , 
 	\end{split}
 \end{align}
 for some constant $\mathfrak C_1>0$ independent of $t,\rho$. Therefore \eqref{Lemma:Delta:1} is proved.

\end{proof}
\section{Proof of Theorem \ref{maintheorem}}
The conclusion of the Theorem follows from Propositions \ref{Lemma:WeakFormulationb}, \ref{Propo:Selfsimilar}, \ref{Lemma:Cascade}.
 \bibliographystyle{plain}
\section{Appendix}
In this Appendix, we give the proof for Lemma \ref{Lemma:Close}. We bound, for a.e. $\widehat k\in\mathbb{S}$, $M\in \mathbb{S}$, $M\ge -1$
\begin{equation}\label{Appendix:1}
\begin{aligned}
	\int_{[0,\infty)\backslash\bigcup_{j=-1}^M\Theta_j}\mathrm{d}|k|f_n(|k|\widehat{k})\ = \ &	\int_{\bigcup_{j=M+1}^\infty\Theta_j}\mathrm{d}|k|f_n(|k|\widehat{k}) \ \le \ & R\mathfrak{C}^{-M-1}\frac{1}{1-\mathfrak C}.
\end{aligned}
\end{equation}
Taking the limit $n\to\infty$, we bound
\begin{equation}\label{Appendix:2}
	\begin{aligned}
		\int_{[0,\infty)\backslash\bigcup_{j=-1}^M\Theta_j}\mathrm{d}|k|f(|k|\widehat{k}) \ \le \ & R\mathfrak{C}^{-M-1}\frac{1}{1-\mathfrak C},
	\end{aligned}\end{equation}
which, after taking the limit $M\to\infty$, implies
\begin{equation}\label{Appendix:3}
	\begin{aligned}
		\int_{[0,\infty)\backslash\bigcup_{j=-1}^\infty\Theta_j}\mathrm{d}|k|f(|k|\widehat{k}) \ = \ & 0.
\end{aligned}\end{equation}
Moreover, since
\begin{equation}\label{Appendix:2}
	\begin{aligned}
		\int_{\Theta_M}\mathrm{d}|k|f_n(|k|\widehat{k}) \ \le \ & R\mathfrak{C}^{-M},\ \ \forall M\ge0,\\
	\mbox{ and             }	\int_{\{|k|=0\}}\mathrm{d}|k|f_n(|k|\widehat{k}) \ \le \ & R,
\end{aligned}\end{equation}
 it is clear that $\|f\|_\mathfrak{S}\le R$. Therefore, $f\in A$. 
\bibliography{WaveTurbulence}

\def\cprime{$'$} \def\cprime{$'$} \def\cprime{$'$} \def\cprime{$'$}
  \def\cprime{$'$} \def\cprime{$'$} \def\cprime{$'$}
\begin{thebibliography}{10}

\bibitem{das2024numerical}
Das A. and M.-B. Tran.
\newblock Numerical schemes for a fully nonlinear coagulation-fragmentation
  model coming from wave kinetic theory.
\newblock {\em arXiv preprint arXiv:2412.05402}, 2024.

\bibitem{AlonsoGambaBinh}
R.~Alonso, I.~M. Gamba, and M.-B. Tran.
\newblock The {C}auchy problem and {BEC} stability for the quantum
  {B}oltzmann-{G}ross-{P}itaevskii system for bosons at very low temperature.
\newblock {\em arXiv preprint arXiv:1609.07467}, 2016.

\bibitem{ampatzoglou2024inhomogeneous}
I.~Ampatzoglou, J.~K. Miller, N.~Pavlovi{\'c}, and M.~Taskovi{\'c}.
\newblock Inhomogeneous wave kinetic equation and its hierarchy in polynomially
  weighted linfty spaces.
\newblock {\em arXiv preprint arXiv:2405.03984}, 2024.

\bibitem{WiemanCornell}
M.H. Anderson, J.R. Ensher, M.R. Matthews, C.E. Wieman, and E.A. Cornell.
\newblock Observation of {B}ose-{E}instein {C}ondensation in a dilute atomic
  vapor.
\newblock {\em Science}, 269(5221):198--201, 1995.

\bibitem{bishop1978study}
D.~J. Bishop and J.~D. Reppy.
\newblock Study of the superfluid transition in two-dimensional he 4 films.
\newblock {\em Physical Review Letters}, 40(26):1727, 1978.

\bibitem{bradley1995evidence}
C.~C. Bradley, C.~A. Sackett, J.~J. Tollett, and R.~G. Hulet.
\newblock Evidence of {B}ose-{E}instein condensation in an atomic gas with
  attractive interactions.
\newblock {\em Physical review letters}, 75(9):1687, 1995.

\bibitem{brezis2011functional}
H.~Brezis.
\newblock {\em Functional analysis, Sobolev spaces and partial differential
  equations}, volume~2.
\newblock Springer, 2011.

\bibitem{buckmaster2019onthe}
T.~Buckmaster, P.~Germain, Z.~Hani, and J.~Shatah.
\newblock On the kinetic wave turbulence description for {NLS}.
\newblock {\em Quart. Appl. Math.}, 78(2):261--275, 2020.

\bibitem{buckmaster2019onset}
T.~Buckmaster, P.~Germain, Z.~Hani, and J.~Shatah.
\newblock {\em Inventiones mathematicae}, 225:787--855, 2021.

\bibitem{collot2024stability}
C.~Collot, H.~Dietert, and P.~Germain.
\newblock Stability and cascades for the {K}olmogorov--{Z}akharov spectrum of
  wave turbulence.
\newblock {\em Archive for Rational Mechanics and Analysis}, 248(1):7, 2024.

\bibitem{collot2019derivation}
C.~Collot and P.~Germain.
\newblock On the derivation of the homogeneous kinetic wave equation.
\newblock {\em arXiv preprint arXiv:1912.10368}, 2019.

\bibitem{collot2020derivation}
C.~Collot and P.~Germain.
\newblock Derivation of the homogeneous kinetic wave equation: longer time
  scales.
\newblock {\em arXiv preprint arXiv:2007.03508}, 2020.

\bibitem{cortes2020system}
E.~Cort{\'e}s and M.~Escobedo.
\newblock On a system of equations for the normal fluid-condensate interaction
  in a bose gas.
\newblock {\em Journal of Functional Analysis}, 278(2):108315, 2020.

\bibitem{CraciunBinh}
G.~Craciun and M.-B. Tran.
\newblock A reaction network approach to the convergence to equilibrium of
  quantum {B}oltzmann equations for {B}ose gases.
\newblock {\em ESAIM: Control, Optimisation and Calculus of Variations}, 2021.

\bibitem{davis1995bose}
K.~B. Davis, M.-O. Mewes, M.~R. Andrews, N.~J. van Druten, D.~S. Durfee, D.~M.
  Kurn, and W.~Ketterle.
\newblock Bose-{E}instein condensation in a gas of sodium atoms.
\newblock {\em Physical review letters}, 75(22):3969, 1995.

\bibitem{deng2021propagation}
Y.~Deng and Z.~Hani.
\newblock Derivation of the wave kinetic equation: full range of scaling laws.
\newblock {\em arXiv preprint arXiv:2110.04565}, 2021.

\bibitem{deng2023full}
Y.~Deng and Z.~Hani.
\newblock Full derivation of the wave kinetic equation.
\newblock {\em Inventiones mathematicae}, 233(2):543--724, 2023.

\bibitem{deng2023long}
Y.~Deng and Z.~Hani.
\newblock Long time justification of wave turbulence theory.
\newblock {\em arXiv preprint arXiv:2311.10082}, 2023.

\bibitem{deng2022wave}
Yu~Deng, Alexandru~D Ionescu, and Fabio Pusateri.
\newblock On the wave turbulence theory of 2d gravity waves, i: deterministic
  energy estimates.
\newblock {\em Communications on Pure and Applied Mathematics}, 78(2):211--322,
  2025.

\bibitem{dolce2024convergence}
M.~Dolce and R.~Grande.
\newblock On the convergence rates of discrete solutions to the wave kinetic
  equation.
\newblock {\em Mathematics In Engineering}, 6(4):536--558, 2024.

\bibitem{dymov2021large}
A.~Dymov, S.~Kuksin, A.~Maiocchi, and S.~Vladuts.
\newblock The large-period limit for equations of discrete turbulence.
\newblock {\em arXiv preprint arXiv:2104.11967}, 2021.

\bibitem{dymov2020zakharov}
Andrey Dymov and Sergei. Kuksin.
\newblock On the {Z}akharov-{L'vov} stochastic model for wave turbulence.
\newblock In {\em Doklady Mathematics}, volume 101, pages 102--109. Springer,
  2020.

\bibitem{dymov2019formal}
Andrey Dymov and Sergei Kuksin.
\newblock Formal expansions in stochastic model for wave turbulence 1:
  {K}inetic limit.
\newblock {\em Comm. Math. Phys.}, 382(2):951--1014, 2021.

\bibitem{dymov2019formal2}
Andrey Dymov and Sergei Kuksin.
\newblock Formal {E}xpansions in {S}tochastic {M}odel for {W}ave {T}urbulence
  2: {M}ethod of {D}iagram {D}ecomposition.
\newblock {\em J. Stat. Phys.}, 190(1):Paper No. 3, 2023.

\bibitem{escobedo2023linearized1}
M.~Escobedo.
\newblock On the linearized system of equations for the condensate--normal
  fluid interaction at very low temperature.
\newblock {\em Studies in Applied Mathematics}, 150(2):448--456, 2023.

\bibitem{escobedo2023linearized}
M.~Escobedo.
\newblock On the linearized system of equations for the condensate-normal fluid
  interaction near the critical temperature.
\newblock {\em Archive for Rational Mechanics and Analysis}, 247(5):92, 2023.

\bibitem{escobedo2024instability}
M.~Escobedo and A.~Menegaki.
\newblock Instability of singular equilibria of a wave kinetic equation.
\newblock {\em arXiv preprint arXiv:2406.05280}, 2024.

\bibitem{EscobedoBinh}
M.~Escobedo and M.-B. Tran.
\newblock Convergence to equilibrium of a linearized quantum {B}oltzmann
  equation for bosons at very low temperature.
\newblock {\em Kinetic and Related Models}, 8(3):493--531, 2015.

\bibitem{EscobedoVelazquez:2015:FTB}
M.~Escobedo and J.~J.~L. Vel{\'a}zquez.
\newblock Finite time blow-up and condensation for the bosonic {N}ordheim
  equation.
\newblock {\em Invent. Math.}, 200(3):761--847, 2015.

\bibitem{EscobedoVelazquez:2015:OTT}
M.~Escobedo and J.~J.~L. Vel{\'a}zquez.
\newblock On the theory of weak turbulence for the nonlinear {S}chr\"odinger
  equation.
\newblock {\em Mem. Amer. Math. Soc.}, 238(1124):v+107, 2015.

\bibitem{EPV}
Miguel Escobedo, Federica Pezzotti, and Manuel Valle.
\newblock Analytical approach to relaxation dynamics of condensed {B}ose gases.
\newblock {\em Ann. Physics}, 326(4):808--827, 2011.

\bibitem{fisher1988dilute}
D.~S. Fisher and P.~C. Hohenberg.
\newblock Dilute bose gas in two dimensions.
\newblock {\em Physical Review B}, 37(10):4936, 1988.

\bibitem{GambaSmithBinh}
I.~M. Gamba, L.~M. Smith, and M.-B. Tran.
\newblock On the wave turbulence theory for stratified flows in the ocean.
\newblock {\em M3AS: Mathematical Models and Methods in Applied Sciences. Vol.
  30, No. 1 105-137}, 2020.

\bibitem{germain2017optimal}
P.~Germain, A.~D. Ionescu, and M.-B. Tran.
\newblock Optimal local well-posedness theory for the kinetic wave equation.
\newblock {\em Journal of Functional Analysis}, 279(4):108570, 2020.

\bibitem{germain2024stability}
P.~Germain, J.~La, and A.~Menegaki.
\newblock Stability of rayleigh-jeans equilibria in the kinetic fpu equation.
\newblock {\em arXiv preprint arXiv:2409.01507}, 2024.

\bibitem{germain2023local}
P.~Germain, J.~La, and K.~Z. Zhang.
\newblock Local well-posedness for the kinetic mmt model.
\newblock {\em arXiv preprint arXiv:2310.11893}, 2023.

\bibitem{germain2025local}
P.~Germain, J.~La, and K.~Z. Zhang.
\newblock Local well-posedness for the kinetic mmt model.
\newblock {\em Communications in Mathematical Physics}, 406(1):1--38, 2025.

\bibitem{germain2024universality}
P.~Germain and H.~Zhu.
\newblock On universality for the kinetic wave equation.
\newblock {\em arXiv preprint arXiv:2402.14773}, 2024.

\bibitem{gorlitz2001realization}
A.~G{\"o}rlitz, J.~M. Vogels, A.~E. Leanhardt, C.~Raman, T.~L. Gustavson, J.~R.
  Abo-Shaeer, A.~P. Chikkatur, S.~Gupta, S.~Inouye, T.~Rosenband, et~al.
\newblock Realization of bose-einstein condensates in lower dimensions.
\newblock {\em Physical review letters}, 87(13):130402, 2001.

\bibitem{grande2024rigorous}
R.~Grande and Z.~Hani.
\newblock Rigorous derivation of damped-driven wave turbulence theory.
\newblock {\em arXiv preprint arXiv:2407.10711}, 2024.

\bibitem{GriffinNikuniZaremba:BCG:2009}
A.~Griffin, T.~Nikuni, and E.~Zaremba.
\newblock {\em Bose-condensed gases at finite temperatures}.
\newblock Cambridge University Press, Cambridge, 2009.

\bibitem{ReichlGust:2012:CII}
E.~D Gust and L.~E. Reichl.
\newblock Collision integrals in the kinetic equations ofdilute bose-einstein
  condensates.
\newblock {\em arXiv:1202.3418}, 2012.

\bibitem{hani2023inhomogeneous}
Z.~Hani, J.~Shatah, and H.~Zhu.
\newblock Inhomogeneous turbulence for wick {NLS}.
\newblock {\em Communications on Pure and Applied Mathematics},
  77(11):4100--4162, 2024.

\bibitem{hannani2022wave}
A.~Hannani, M.~Rosenzweig, G.~Staffilani, and M.-B. Tran.
\newblock On the wave turbulence theory for a stochastic kdv type
  equation--generalization for the inhomogeneous kinetic limit.
\newblock {\em arXiv preprint arXiv:2210.17445}, 2022.

\bibitem{hohenberg1967existence}
P.~C. Hohenberg.
\newblock Existence of long-range order in one and two dimensions.
\newblock {\em Physical Review}, 158(2):383, 1967.

\bibitem{KD1}
T.~R. Kirkpatrick and J.~R. Dorfman.
\newblock Transport theory for a weakly interacting condensed {B}ose gas.
\newblock {\em Phys. Rev. A (3)}, 28(4):2576--2579, 1983.

\bibitem{KD3}
T.~R. Kirkpatrick and J.~R. Dorfman.
\newblock Transport coefficients in a dilute but condensed bose gas.
\newblock {\em J. Low Temp. Phys.}, 58:399--415, 1985.

\bibitem{KD2}
T.~R. Kirkpatrick and J.~R. Dorfman.
\newblock Transport in a dilute but condensed nonideal bose gas: Kinetic
  equations.
\newblock {\em J. Low Temp. Phys.}, 58:301--331, 1985.

\bibitem{LukkarinenSpohn:WNS:2011}
J.~Lukkarinen and H.~Spohn.
\newblock Weakly nonlinear {S}chr\"odinger equation with random initial data.
\newblock {\em Invent. Math.}, 183(1):79--188, 2011.

\bibitem{ma2022almost}
X.~Ma.
\newblock Almost sharp wave kinetic theory of multidimensional kdv type
  equations with d $>=$ 3.
\newblock {\em arXiv preprint arXiv:2204.06148}, 2022.

\bibitem{menegaki20222}
A.~Menegaki.
\newblock L2-stability near equilibrium for the 4 waves kinetic equation.
\newblock {\em Kinetic and Related Models}, 17(4):514--532, 2024.

\bibitem{nelson1977universal}
D.~R. Nelson and J.~M. Kosterlitz.
\newblock Universal jump in the superfluid density of two-dimensional
  superfluids.
\newblock {\em Physical Review Letters}, 39(19):1201, 1977.

\bibitem{nguyen2017quantum}
T.~T. Nguyen and M.-B. Tran.
\newblock On the {K}inetic {E}quation in {Z}akharov's {W}ave {T}urbulence
  {T}heory for {C}apillary {W}aves.
\newblock {\em SIAM J. Math. Anal.}, 50(2):2020--2047, 2018.

\bibitem{pavlovic2025inhomogeneous}
N.~Pavlovi{\'c}, M.~Taskovi{\'c}, and L.~Velasco.
\newblock Inhomogeneous six-wave kinetic equation in exponentially weighted
  linfty spaces.
\newblock {\em arXiv preprint arXiv:2501.10565}, 2025.

\bibitem{penrose1956bose}
O.~Penrose and L.~Onsager.
\newblock Bose-einstein condensation and liquid helium.
\newblock {\em Physical Review}, 104(3):576, 1956.

\bibitem{PomeauBinh}
Y.~Pomeau and M.-B. Tran.
\newblock Statistical physics of non equilibrium quantum phenomena.
\newblock {\em Lecture Notes in Physics, Springer}, 2019.

\bibitem{reichl2016modern}
L.~E. Reichl.
\newblock {\em A Modern Course in Statistical Physics}.
\newblock John Wiley \& Sons, 2016.

\bibitem{reichl2019kinetic}
L.~E. Reichl and M.-B. Tran.
\newblock A kinetic equation for ultra-low temperature bose--einstein
  condensates.
\newblock {\em Journal of Physics A: Mathematical and Theoretical},
  52(6):063001, 2019.

\bibitem{rudin1974real}
W.~Rudin.
\newblock Real and complex analysis.
\newblock 1974.

\bibitem{rumpf2021wave}
B.~Rumpf, A.~Soffer, and M.-B. Tran.
\newblock On the wave turbulence theory: ergodicity for the elastic beam wave
  equation.
\newblock {\em arXiv preprint arXiv:2108.13223}, 2021.

\bibitem{rychtarik2004two}
D.~Rychtarik, B.~Engeser, H.-C. N{\"a}gerl, and R.~Grimm.
\newblock Two-dimensional bose-einstein condensate in an optical surface trap.
\newblock {\em Physical review letters}, 92(17):173003, 2004.

\bibitem{soffer2018dynamics}
A.~Soffer and M.-B. Tran.
\newblock On the dynamics of finite temperature trapped bose gases.
\newblock {\em Advances in Mathematics}, 325:533--607, 2018.

\bibitem{soffer2020energy}
A.~Soffer and M.-B. Tran.
\newblock On the energy cascade of 3-wave kinetic equations: beyond
  kolmogorov--zakharov solutions.
\newblock {\em Communications in Mathematical Physics}, 376(3):2229--2276,
  2020.

\bibitem{staffilani2021wave}
G.~Staffilani and M.-B. Tran.
\newblock On the wave turbulence theory for a stochastic {KdV} type equation.
\newblock {\em arXiv preprint arXiv:2106.09819}, 2021.

\bibitem{staffilani2024condensation}
G.~Staffilani and M.-B. Tran.
\newblock Condensation and non-condensation times for 4-wave kinetic equations.
\newblock {\em arXiv preprint arXiv:2407.18533}, 2024.

\bibitem{staffilani2024energy}
G.~Staffilani and M.-B. Tran.
\newblock On the energy transfer towards large values of wavenumbers for
  solutions of 4-wave kinetic equations.
\newblock {\em arXiv preprint arXiv:2407.18508}, 2024.

\bibitem{tran2020reaction}
M.-B. Tran, G.~Craciun, L.~M. Smith, and S.~Boldyrev.
\newblock A reaction network approach to the theory of acoustic wave
  turbulence.
\newblock {\em Journal of Differential Equations}, 269(5):4332--4352, 2020.

\bibitem{tran2019boltzmann}
M.-B. Tran and Y.~Pomeau.
\newblock Boltzmann-type collision operators for {B}ogoliubov excitations of
  {B}ose-{E}instein condensates: {A} unified framework.
\newblock {\em Physical Review E 101 (3), 032119}, 2020.

\bibitem{tran2021thermal}
M.-B. Tran and Y.~Pomeau.
\newblock On a thermal cloud-bose--einstein condensate coupling system.
\newblock {\em The European Physical Journal Plus}, 136(5):1--11, 2021.

\bibitem{walton2023numerical}
S.~Walton and M.-B. Tran.
\newblock A numerical scheme for wave turbulence: 3-wave kinetic equations.
\newblock {\em SIAM Journal on Scientific Computing}, 45(4):B467--B492, 2023.

\bibitem{walton2024numerical}
S.~Walton and M.-B. Tran.
\newblock Numerical schemes for 3-wave kinetic equations: A complete treatment
  of the collision operator.
\newblock {\em arXiv preprint arXiv:2402.17481}, 2024.

\bibitem{walton2022deep}
S.~Walton, M.-B. Tran, and A.~Bensoussan.
\newblock A deep learning approximation of non-stationary solutions to wave
  kinetic equations.
\newblock {\em Applied Numerical Mathematics}, 2022.

\bibitem{PomeauBrachetMetensRica}
S.~M'etens Y.~Pomeau, M.A.~Brachet and S.~Rica.
\newblock Th\'eorie cin\'etique d'un gaz de bose dilu\'e avec condensat.
\newblock {\em C. R. Acad. Sci. Paris S'er. IIb M'ec. Phys. Astr.},
  327:791--798, 1999.

\bibitem{ZarembaNikuniGriffin:1999:DOT}
Nikuni T. Griffin~A. Zaremba, E.
\newblock Dynamics of trapped bose gases at finite temperatures.
\newblock {\em J. Low Temp. Phys.}, 116:277--345, 1999.

\end{thebibliography}

\end{document}